\documentclass[11pt,a4paper]{amsart}

\usepackage{amsmath,amstext,amssymb,amsthm,amsfonts,graphicx}
\usepackage[all,cmtip]{xy}

\theoremstyle{plain}
\newtheorem*{theorem*}{Theorem}
\newtheorem*{remark*}{Remark}
\newtheorem*{example*}{Example}
\newtheorem{lemma}{Lemma}
\newtheorem{proposition}[lemma]{Proposition}

\newtheorem{theorem}[lemma]{Theorem}
\newtheorem{conjecture}[lemma]{Conjecture}
\newtheorem{question}[lemma]{Question}
\newtheorem*{conjecture*}{Conjecture}

\theoremstyle{definition}
\newtheorem{definition}[lemma]{Definition}

\newtheorem{example}[lemma]{Example}

\theoremstyle{remark}
\newtheorem{remark}[lemma]{Remark}

\newtheorem{notation}[lemma]{Notation}

\newcommand{\bR}{{\mathbb R}}

\newcommand{\abs}[1]{\left|{#1}\right|}

\def\quotient#1#2{%
    \raise1ex\hbox{$#1$}\Big/\lower1ex\hbox{$#2$}%
}

\begin{document}

\date{October 11, 2021}
\setcounter{tocdepth}{1}
\title[Board games,
random boards and long boards]{Board games, random boards and long boards}
 \author{Ary Shaviv}
\address{Department of Mathematics, Princeton University, Princeton, New Jersey 08544}
\email{ashaviv@math.princeton.edu}

\maketitle

\begin{abstract}
For any odd integer $n\geq3$ a board (of size $n$) is a square array of $n\times n$ positions with a simple rule of how to move between positions. The goal of the game we introduce is to find a path from the upper left corner of a board to the center of the square. If there exists such a path we say that the board is solvable, and we say that the length of this board is the length of a shortest such path. There are $8^{n^2}$ different boards. We discuss various properties of these boards and present some questions and conjectures. In particular, we show that for $n\gg1$ roughly $\frac{1}{3}$ of the boards are solvable, and that the expected length of a random solvable board tends to $\frac{209}{96}$, i.e., very big solvable boards tend to have extremely short solutions.   
\end{abstract}

\tableofcontents

\section{Introduction} In this paper we introduce a very simple board game and study its properties. The exact definitions are given in Section \ref{section_definitions}, however the intuitive idea is quite natural: for an odd integer $n\geq3$ we are given an array of $n\times n$ arrows, where each arrow is an element of the standard 8 directions of the wind compass  (\emph{a board of size $n$}). 

\begin{center}\includegraphics[width=8cm]{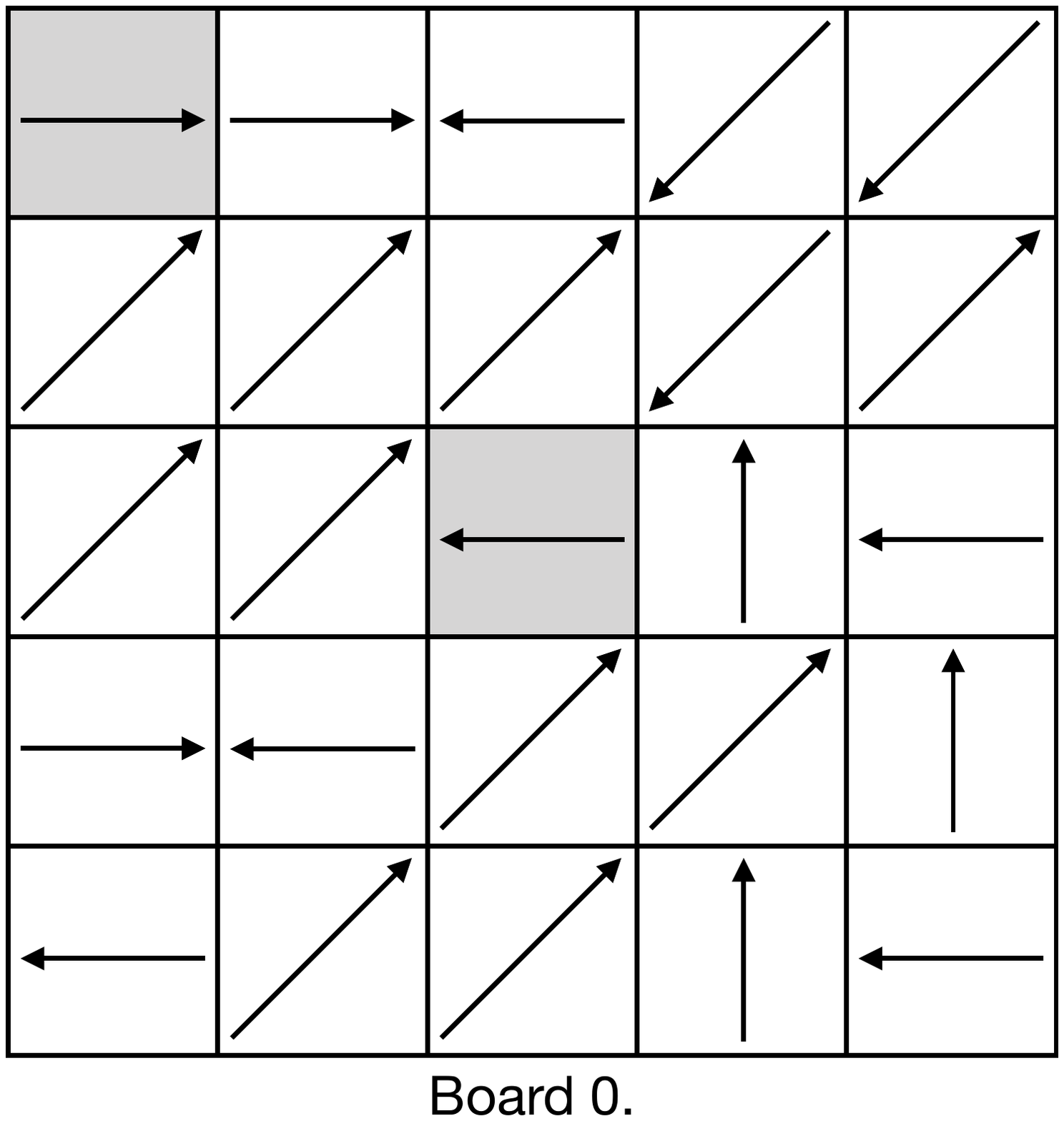}\end{center}

We start at the arrow in the upper left corner, and are allowed to move from the square we are in to any square that the arrow we stand on is pointing to.  Our goal is to get to the middle square of this array (the $(\frac{n+1}{2},\frac{n+1}{2})$ entry) and if we can do it (if the board is \emph{solvable}) we want to do it in the minimal number of steps (this minimal number is called the \emph{length} of the board). For instance, Board 0 above is obviously solvable and its length is 2 -- the quickest way of arriving from the upper left corner to the center is to first jump to the upper right corner and then jump to the center square. This required 2 jumps and so Board 0 is solvable and its length is 2. 

In this paper we ask many questions about these boards and games, give some (somewhat surprising) answers, and state a few conjectures. All of our proofs are elementary, and mainly rely on basic probability theory. Here is a brief survey of the main issues that we study: 

There are obviously $8^{n^2}$ different boards of size $n$. Say we choose at random a board of size $n$ by the uniform discrete measure on the set of all boards of size $n$.
 Clearly the probability of choosing a solvable board cannot be more than $\frac{3}{8}$, as a board cannot be solvable unless its upper left arrow is either $\rightarrow,\searrow$ or $\downarrow$. We show that as $n$ tends to infinity the probability of choosing a solvable board goes to this upper bound $\frac{3}{8}$ (Proposition \ref{Prop_prob_solvable}).

Now assume we choose at random a \emph{solvable} board of size $n$ by the uniform discrete measure on the set of \emph{solvable} boards of size $n$. What is the expected value of the length of such a random board? We show that as $n$ tends to infinity this expected value goes to $\frac{209}{96}$, namely, the expected value
of the length of a big random solvable board tends to be slightly more than
2 (Theorem \ref{theorem_on_expected_length}).

For a fixed $n$ we may ask what is "the worst case scenario", i.e., what is the maximal length of a solvable board of size $n$ (we denote this number by $ML(n)$). Obviously $ML(n)\leq n^2$ however we conjecture that as $n$ tends to infinity $ML(n)=O(n)$ or at least $ML(n)=o(n^2)$ (Conjectures \ref{main_conjecture},\ref{weak_conjecture}). We in particular show that if we place our boards on tori then this conjecture holds: in fact, $\overline{ML}(n)$ -- the maximal length of a board on a torus of size $n$ -- satisfies $\overline{ML}(n)=\Theta(n)$ (Proposition \ref{proposition_on_ML_on_torus}). 

We also explain how to associate a directed graph to any given board, translate the questions of solvability and length to this graph, and discuss the classifications of these graphs, their isomorphism classes and so on.

\subsection*{Structure of this paper}In Section \ref{section_definitions} we define the notions of board, game and length, and give some simple examples. 

Section \ref{section_random_boards} is devoted to defining and discussing random boards and random solvable boards and to proving some results regarding their asymptotic behaviour. 

In Section \ref{section_ML} we formulate rigourously the question of \emph{how complicated can a board of a given size be?} and formulate our main conjecture about the answer. 

In Section \ref{section_graphs} we introduce the graph of a board, discuss the properties of such graphs and ask some questions about their characterization.

Finally, In Section \ref{section_remarks} we explain how to determine the solvability of a given board (and its length if solvable) by a well known search algorithm,  discuss some possible generalizations of our boards, and realize our boards as matrices over $\mathbb{F}_9$.

\subsection*{Acknowledgments}The author is grateful to Noga Alon and Michele Fornea for valuable discussions (see more specific thanks in the text) and to Charles Fefferman for his outstanding continuous support. The author was supported by AFOSR Grant FA9550-18-1-069.

\section{Defining boards and games}\label{section_definitions}
\begin{notation}We denote the set of 8 principle winds of the
compass rose by $\mathcal{D}$: $$\mathcal{D}:=\{\text{N,NE,E,SE,S,SW,W,NW}\},$$or
$$\mathcal{D}=\{\uparrow,\nearrow,\rightarrow,\searrow,\downarrow,\swarrow,\leftarrow,\nwarrow\},$$respectively. We sometimes refer to elements of $\mathcal{D}$ as \emph{directions}.
\end{notation}

\begin{definition}Let $n\geq3$ be an odd integer. An $n\times n$ matrix whose entries are directions is called \emph{a board of size $n$}. The set $\text{Mat}_{n\times n}(\mathcal{D})$ is the set of all boards of size $n$. 

For a fixed board $A$ and $1\leq i,j,i',j'\leq n$, $(i,j)\neq(i',j')$ we say that $a_{ij}$ (the $(i,j)$ entry of $A$) \emph{is directing to} $(i',j')$ if either $$((a_{ij}=\uparrow)\wedge(i>i')\wedge(j=j'))$$or $$((a_{ij}=\nearrow)\wedge(i-i'=j'-j>0))$$ or $$((a_{ij}=\rightarrow)\wedge(i=i')\wedge(j<j'))$$and so on (a total of 8 possible conditions).

\emph{A game} on the board $A$ is the following: the player always starts at the upper left corner of the matrix $A$ (the $(1,1)$ position), and at each turn moves to another entry of $A$ by the following rule:

\begin{center}$\star$ \emph{It is allowed to move from position} $(i,j)$ \emph{to position} $(i',j')$ \emph{iff} $a_{ij}$ \emph{is directing to} $(i',j')$.\end{center}
The game is over the first time the player reaches a position that is either (1) or (2):
\begin{enumerate}
\item the $({\frac{n+1}{2},\frac{n+1}{2}})$
position (the center of the matrix);
\item a position $(i,j)$ such that $a_{ij}$
is not directing to any other position in the matrix.
\end{enumerate}
If a game is over after $k$ turns because the player reached the $({\frac{n+1}{2},\frac{n+1}{2}})$ position, then
we say that the player \emph{won the game after $k$ turns}. If the game is over after $k$ turns because the player reached a position $(i,j)$ such that $a_{ij}$ is not directing to any other position in the matrix (e.g., the $(1,n)$ position with $a_{1n}=\uparrow$), then we say that the player \emph{lost the game after $k$ turns}. We denote by $(i_l,j_l)$ the position in which the player
stayed after the
$l^{\text{th}}$ turn. Then, a game on $A$ that ended after $k$ turns is equivalent to the sequence $$(1,1)=(i_0,j_{0}),(i_1,j_{1}),\dots,(i_k,j_{k}),$$where $(i_k,j_{k})=({\frac{n+1}{2},\frac{n+1}{2}})$ if and only if the player won the game. If there exists a game on the board $A$ that ends with the player winning, then we say that \emph{the board $A$ is solvable}. Otherwise, we say that $A$ in not solvable. For a solvable board $A$, we say that the \emph{length} of $A$ is the minimal number $k$, such that there exists a game on $A$ in which the player wins after $k$ turns.  
\end{definition}

\begin{example}\label{first_example}Board 1 below is solvable and it has length 5, whereas Board 2 is not solvable.  The gray background has no mathematical meaning and is presented only in order to make it easier to spot the starting position and the target position. The colored arrows correspond to the sequence $(1,1)=(i_0,j_{0}),(i_1,j_{1}),\dots,(i_k,j_{k})$ above, i.e., they represent a game.
\begin{center}\includegraphics[width=12cm]{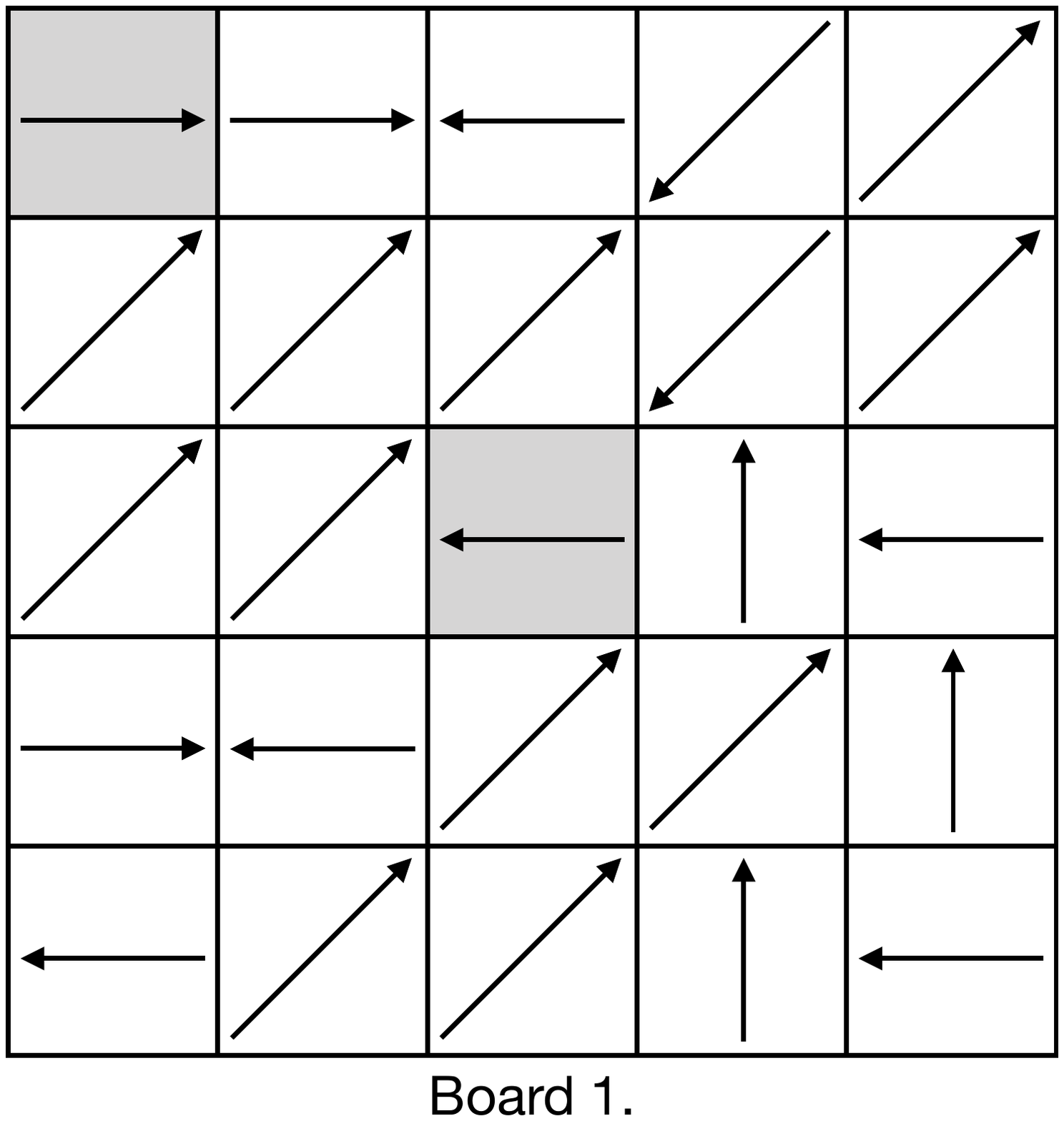}\end{center}
\begin{center}\includegraphics[width=12cm]{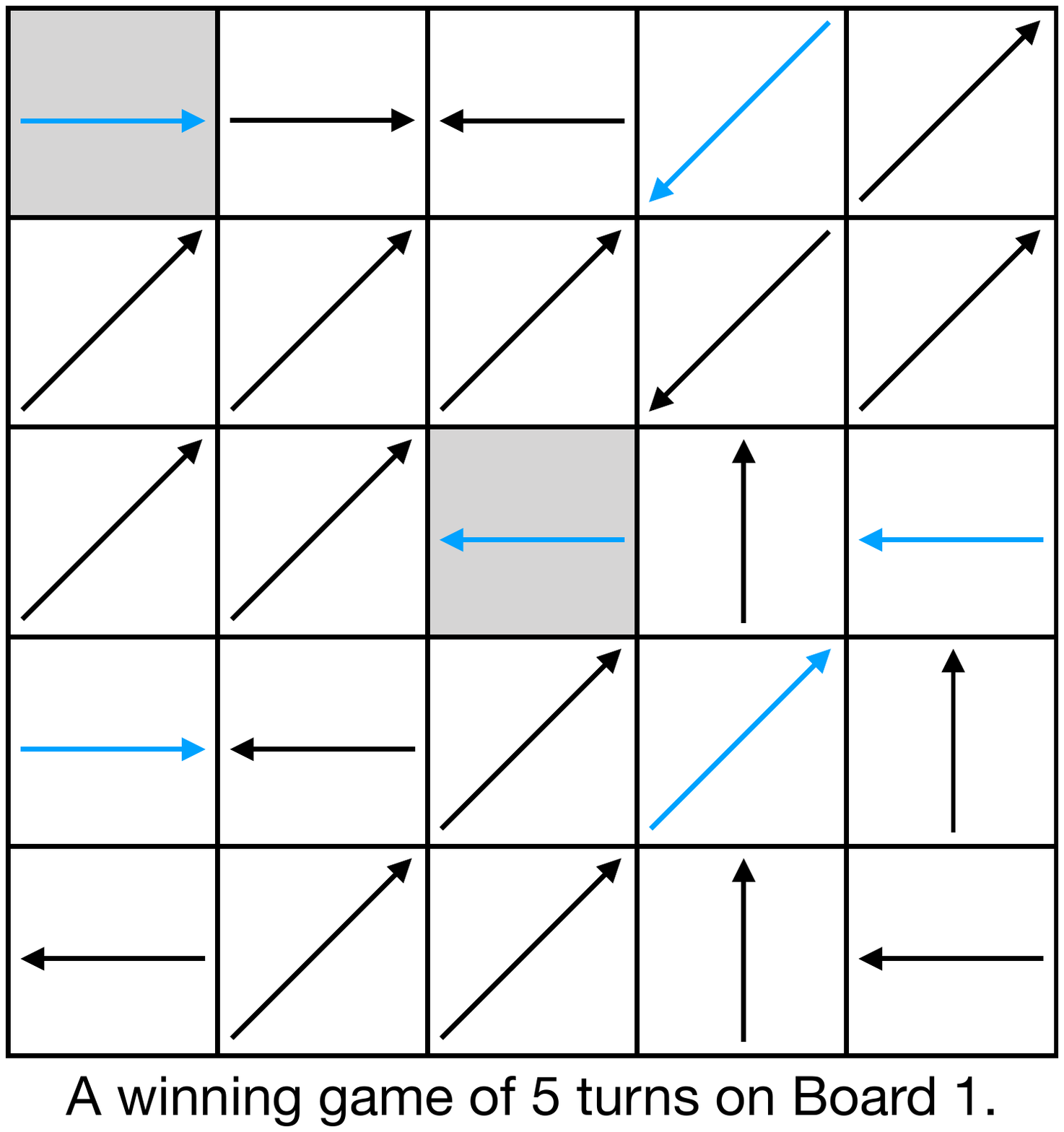}\end{center}
\begin{center}\includegraphics[width=12cm]{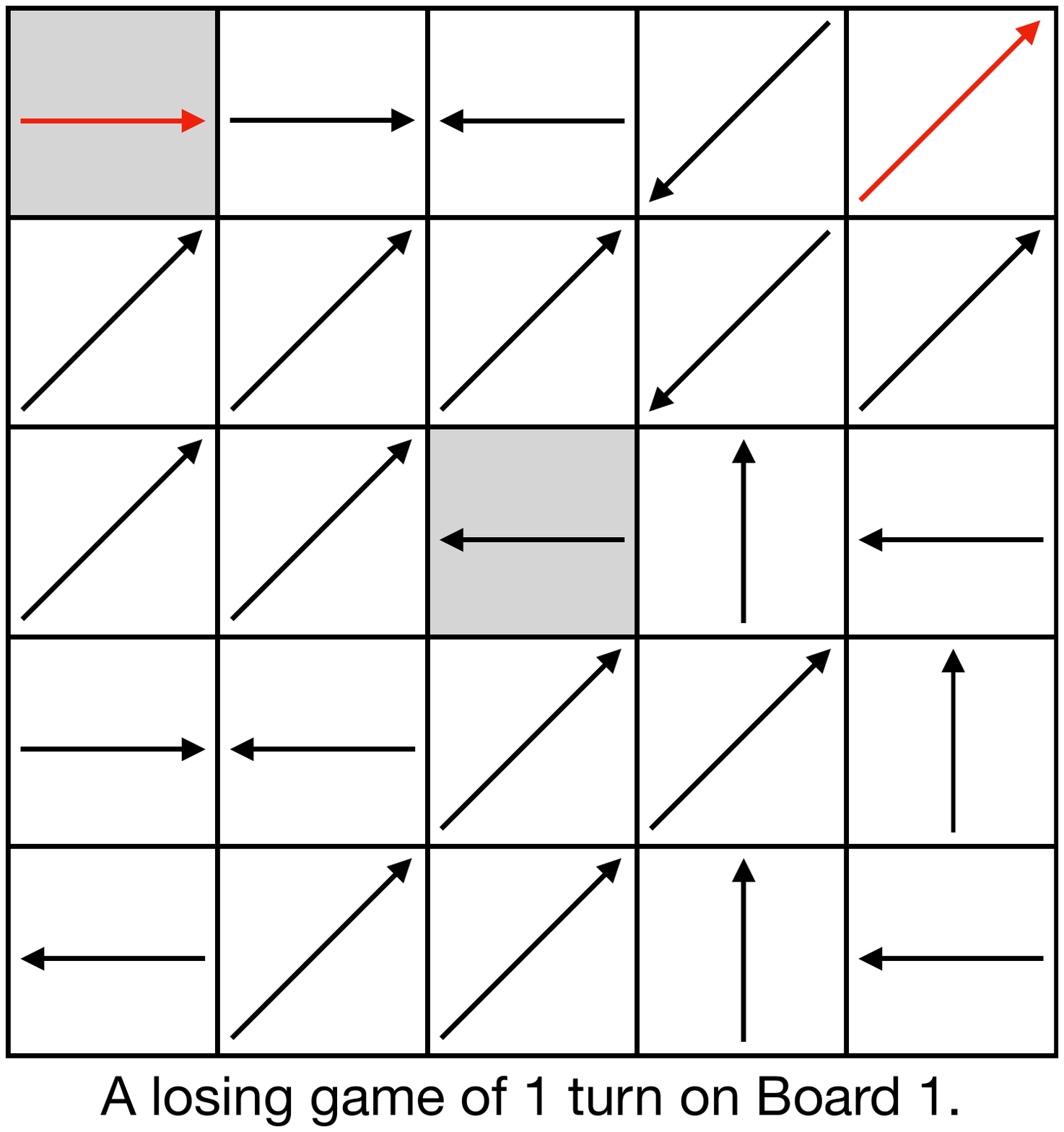}\end{center}
\begin{center}\includegraphics[width=12cm]{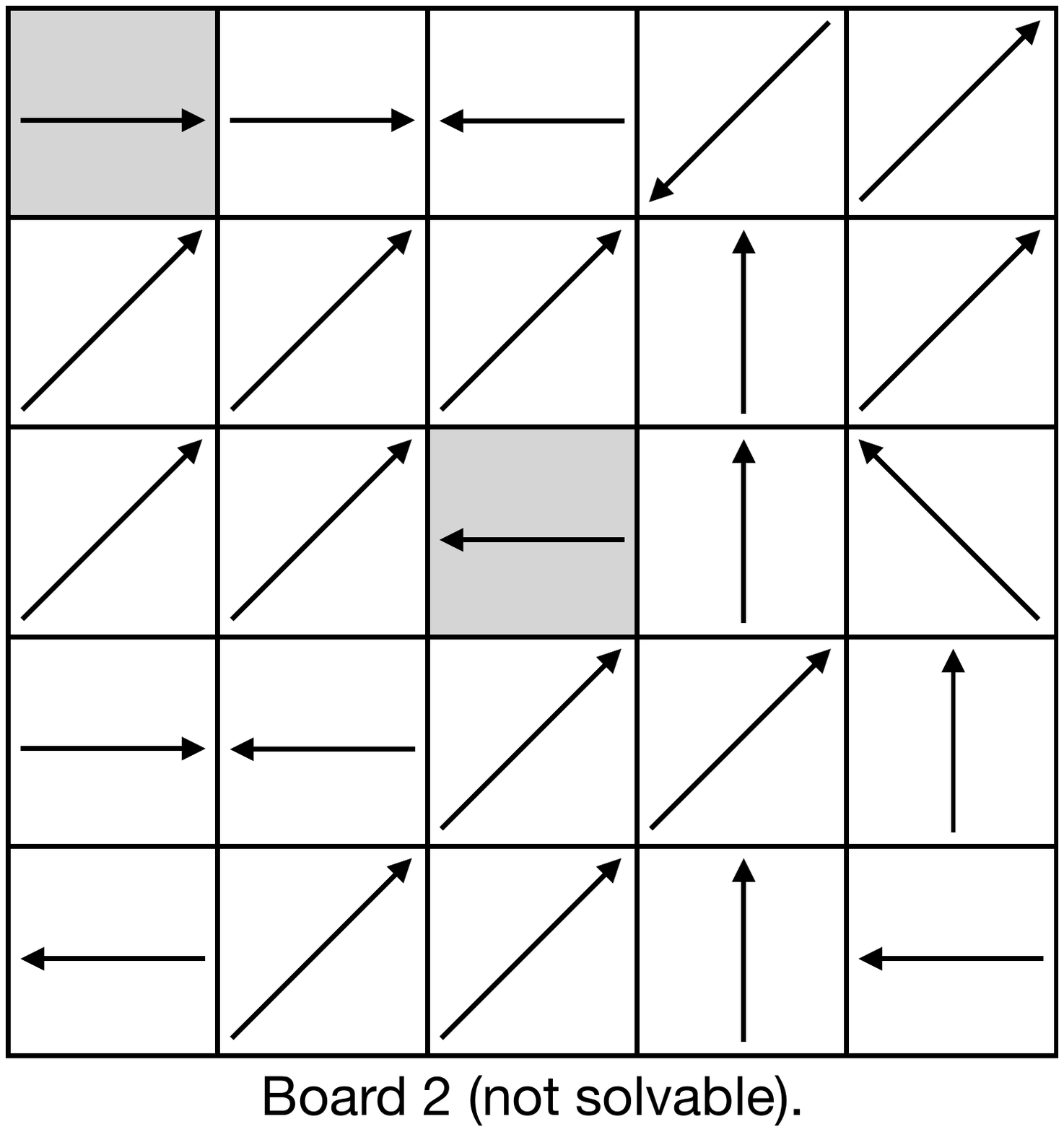}\end{center}
Note that a winning game on a board may be of more steps than the length of the board. For instance, here is a winning game of 6 turns on Board 1 (that has length 5):
\begin{center}\includegraphics[width=12cm]{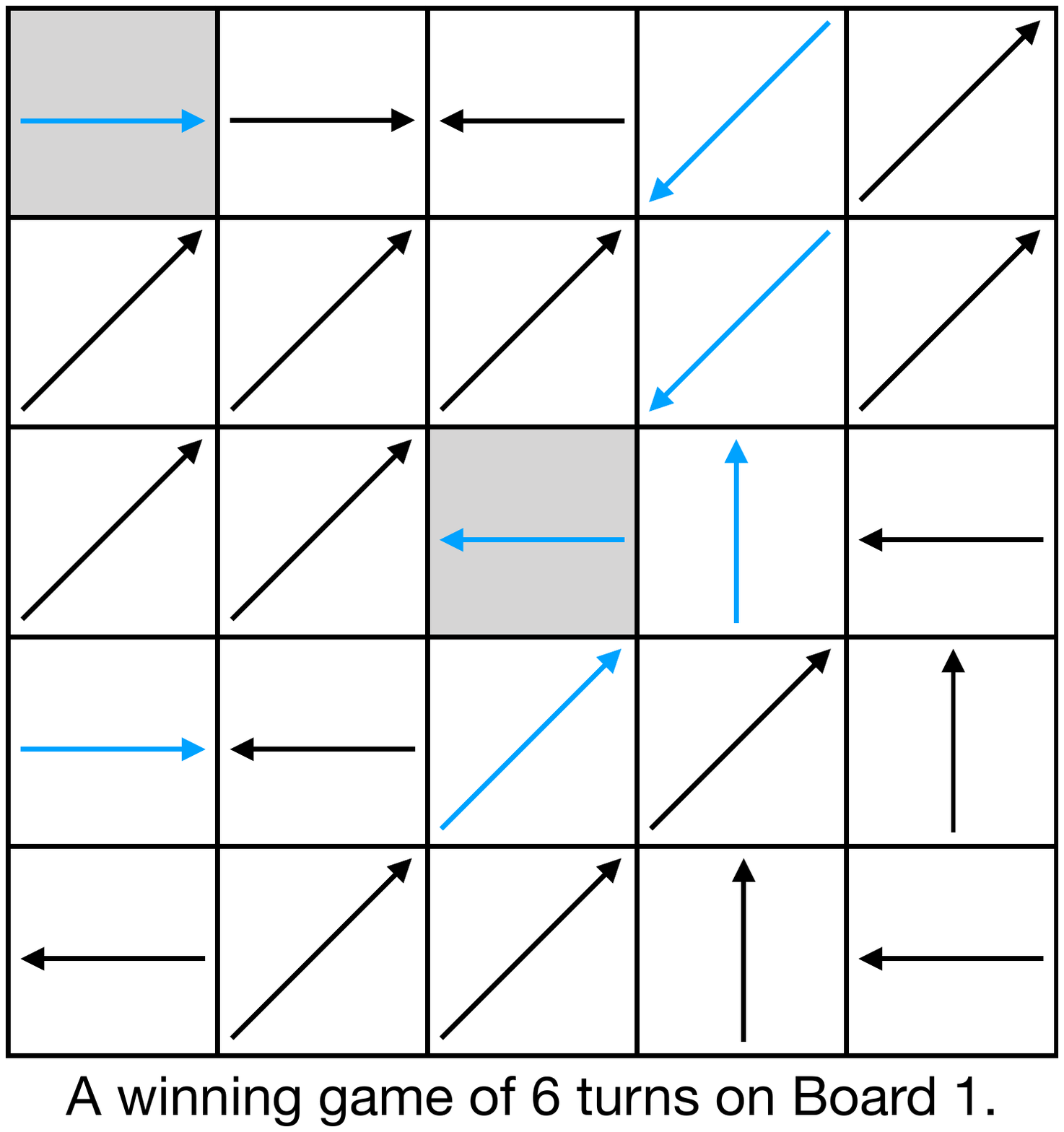}\end{center}
Board 3 below has length 9, but there exists a game of length 24 on it as well. In fact, for any $9\leq k\leq24$ we can easily find a game of length $k$ on this board.
\begin{center}\includegraphics[width=12cm]{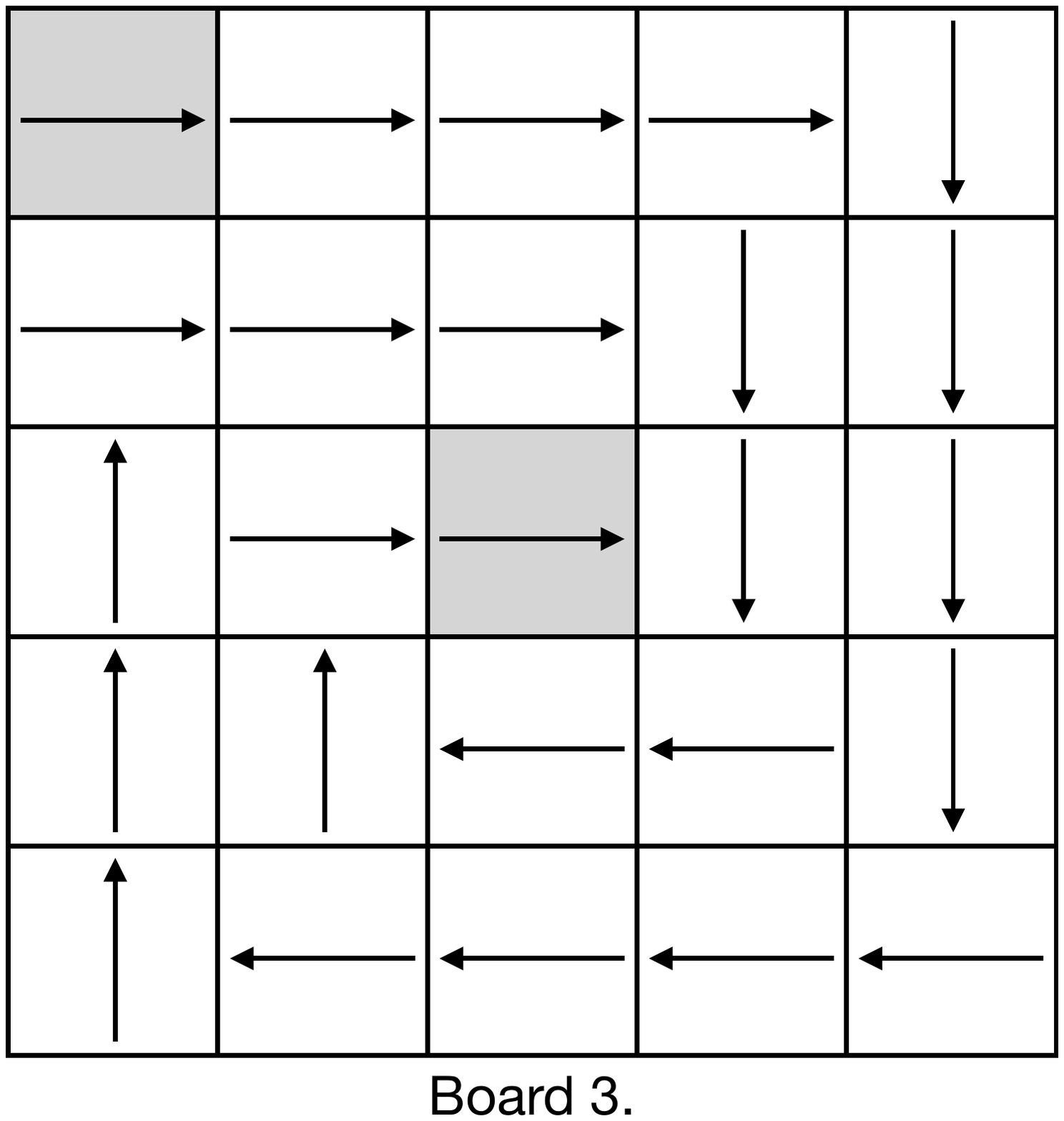}\end{center}
\begin{center}\includegraphics[width=12cm]{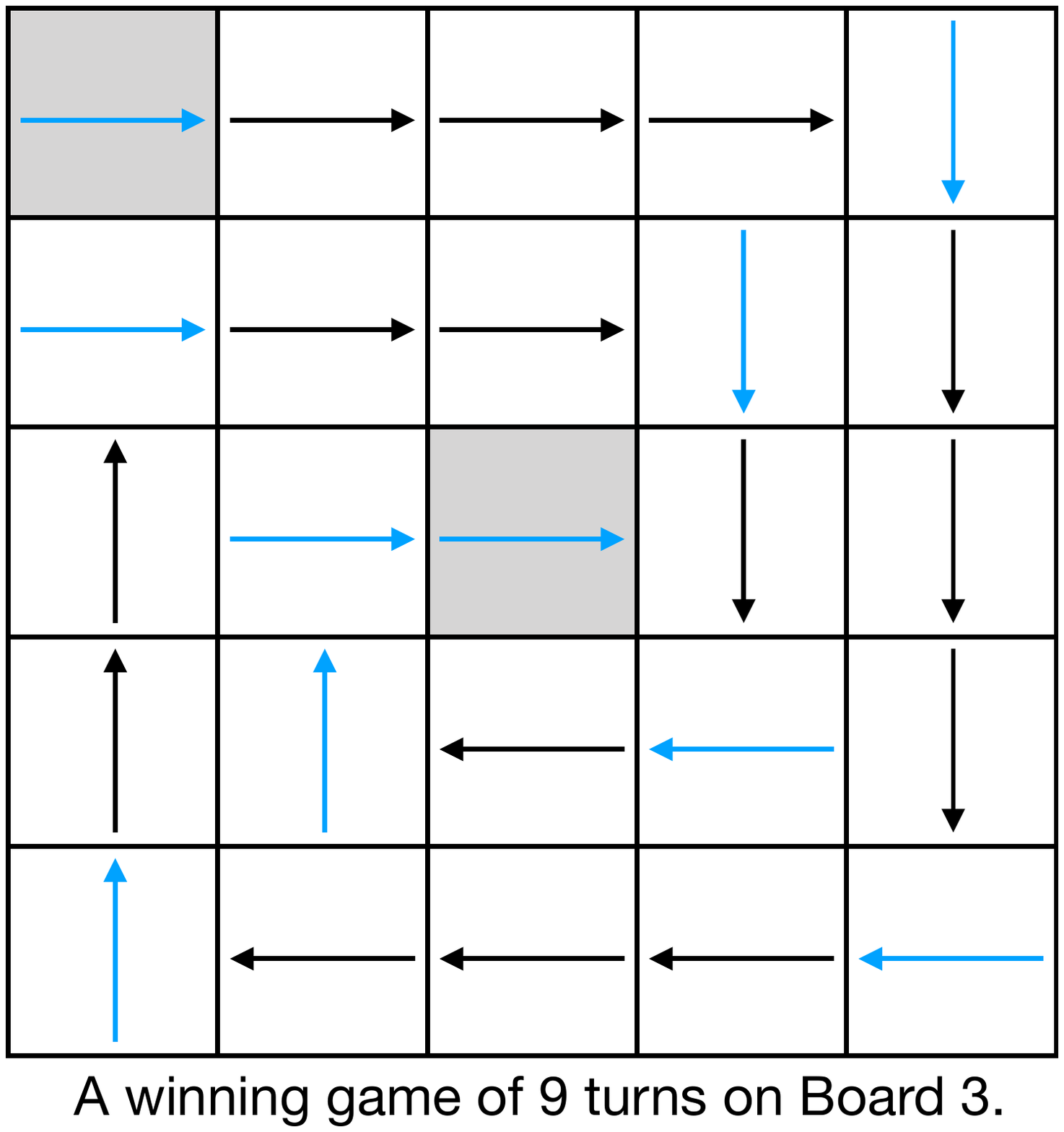}\end{center}
\begin{center}\includegraphics[width=12cm]{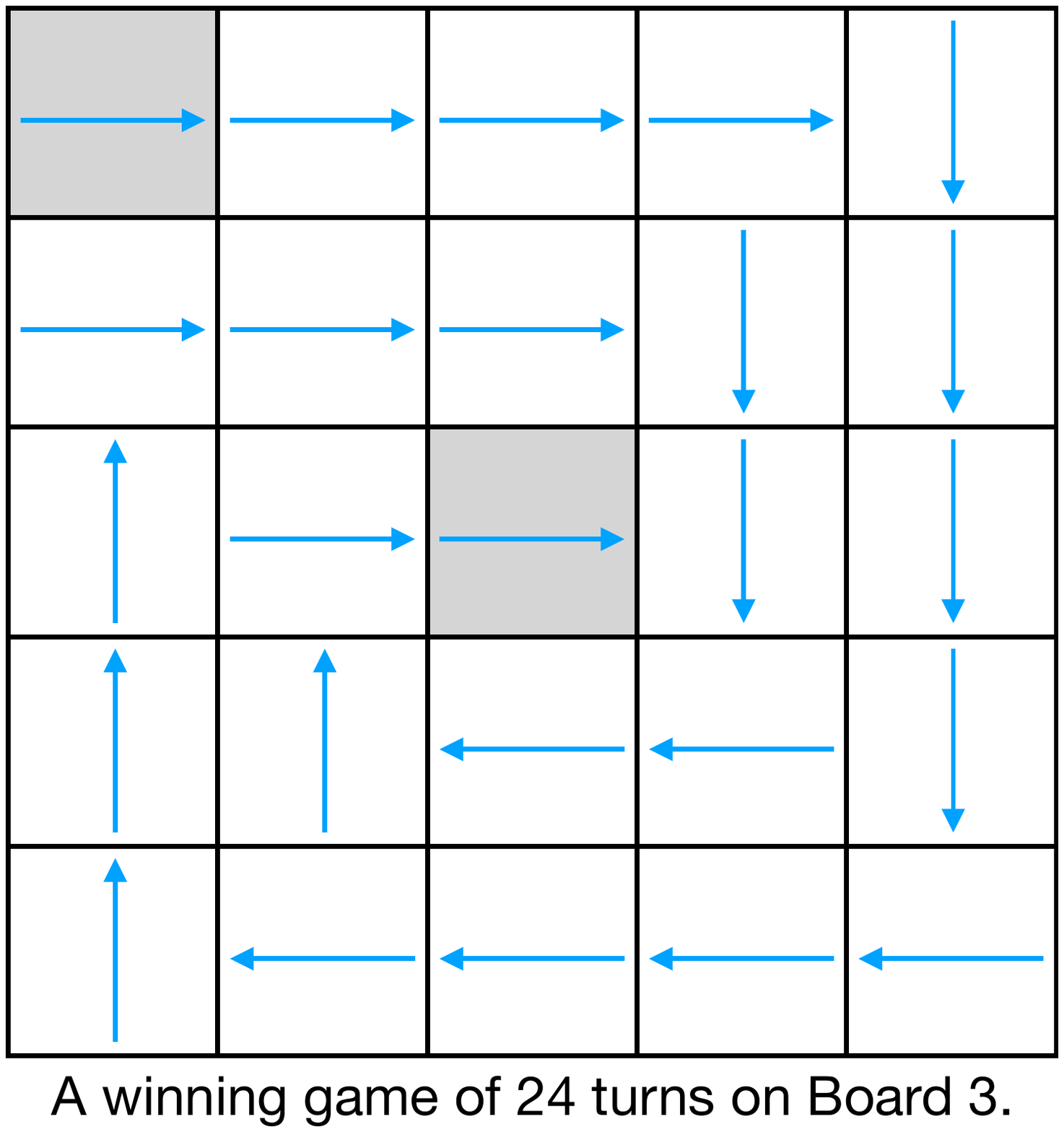}\end{center}

\end{example}

\section{Random boards}\label{section_random_boards}\subsection{The probability of a random board to be solvable}
Let $n\geq3$ be an odd integer. The set $\text{Mat}_{n\times
n}(\mathcal{D})$ is finite (it has cardinality $8^{n^2}$), and so we turn
it into a probability space by assigning each board with equal probability
(the discrete uniform measure). A random board is a board chosen from $\text{Mat}_{n\times
n}(\mathcal{D})$ according to this probability measure. This is of course
equivalent to assigning a direction to each entry in an $n\times n$ matrix
with probability $\frac{1}{8}$ to each direction (and the choices of each
two different entries are independent events). We ask $$\text {Prob}_{A\in\text{Mat}_{n\times
n}(\mathcal{D})}(A\text{ is solvable})=?$$i.e., what is the probability of
a random board to be solvable? 
\subsubsection{Trivial upper bound}\label{trivial_upper_bound_remark}Obviously,
a necessary (but not sufficient) condition for $A$ to be solvable is $a_{11}\in\{\rightarrow,\searrow,\downarrow\}$,
and so $\text
{Prob}_{A\in\text{Mat}_{n\times
n}(\mathcal{D})}(A\text{ is solvable})<\frac{3}{8}$.

\begin{proposition}[proved together with Noga Alon]\label{Prop_prob_solvable}$\lim\limits_{n\to\infty}\text {Prob}_{A\in\text{Mat}_{n\times
n}(\mathcal{D})}(A\text{
is solvable})=\frac{3}{8}$\footnote{Whenever we write limits of the form
$n\to\infty$ in this paper we mean that $n$ runs over the \emph{odd} natural
numbers.}.
\end{proposition}

\begin{proof}To simplify the notation we write
$P_{n}$ instead of $\text {Prob}_{A\in\text{Mat}_{n\times
n}(\mathcal{D})}$.

$$P_{n}(A\text{ is solvable})=$$$$P_{n}(a_{11}=\rightarrow)\cdot
P_{n}(A\text{ is solvable}|a_{11}=\rightarrow)+$$$$P_{n}(a_{11}=\searrow)\cdot
P_{n}(A\text{
is solvable}|a_{11}=\searrow)+$$$$P_{n}(a_{11}=\downarrow)\cdot P_{n}(A\text{
is solvable}|a_{11}=\downarrow)=$$$$\frac{1}{8}[P_{n}(A\text{
is solvable}|a_{11}=\rightarrow)+P_{n}(A\text{
is solvable}|a_{11}=\searrow)+$$$$P_{n}(a_{11}=\downarrow)\cdot P_{n}(A\text{
is solvable}|a_{11}=\downarrow)].$$Thus, it is enough to prove that $$\lim\limits_{n\to\infty}P_{n}(A\text{
is solvable}|a_{11}=\rightarrow)=1,$$ $$\lim\limits_{n\to\infty}P_{n}(A\text{
is solvable}|a_{11}=\downarrow)=1$$ and $$\lim\limits_{n\to\infty}P_{n}(A\text{
is solvable}|a_{11}=\searrow)=1.$$ 
The latter is clear: if $a_{11}=\searrow$ then the sequence $(1,1),({\frac{n+1}{2},\frac{n+1}{2}})$
is always a winning game, and so $P(A\text{
is solvable}|a_{11}=\searrow)=1$.

Assume that $a_{11}=\rightarrow$. In this case we will define $n-1$ "easy
winning games" that may be possible or impossible on the board $A$. We denote
these games by $G_2,G_3,\dots ,G_n$.

For $2\leq i\leq n,i\neq\frac{n+1}{2}$ we define the game $G_i$ by the sequence
$(1,1),(1,i),(\frac{n+1}{2},i),({\frac{n+1}{2},\frac{n+1}{2}})$;
  
For $i=\frac{n-3}{2}+1=\frac{n+1}{2}$ we define the game $G_i$ by the sequence
$(1,1),(1,\frac{n+1}{2}),({\frac{n+1}{2},\frac{n+1}{2}})$;

\begin{center}\includegraphics[width=15cm]{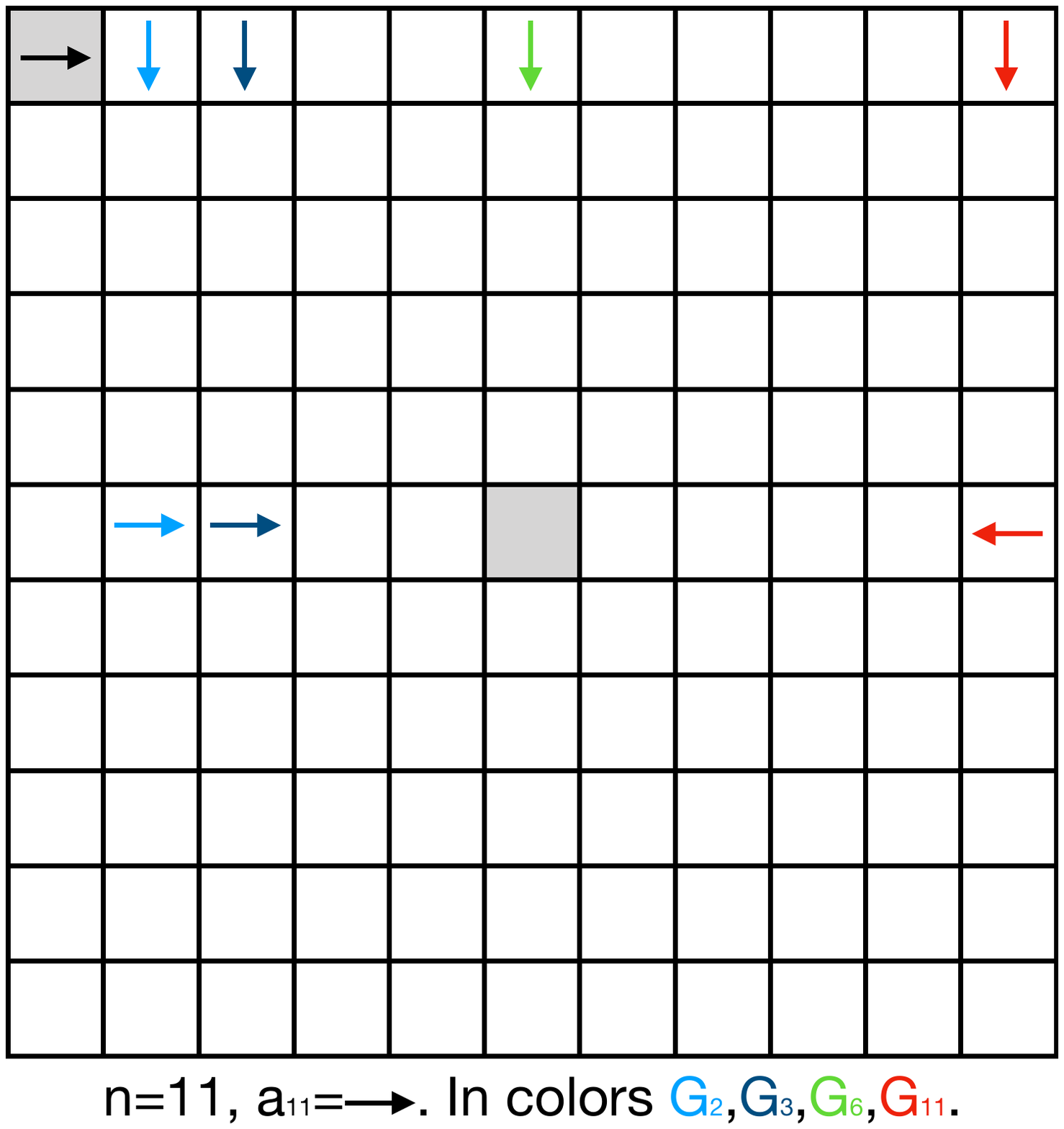}\end{center}
Note that for $i\neq \frac{n+1}{2}$ the probability of the game $G_i$ to
be possible on $A$ is $\frac{1}{64}$ (recall $a_{11}=\rightarrow$), and that
the probability of $G_{\frac{n+1}{2}}$ to be possible on $A$ is $\frac{1}{8}$.
Moreover, the events $$\{G_i\text{ is possible on }A\text{ given }a_{11}=\rightarrow\}_{i=2}^n$$
are independent, and so $$P_{n}(A\text{
is solvable}|a_{11}=\rightarrow)\geq 1-\prod\limits_{i=2}^{n}P_{n}(\neg (G_i\text{
is possible in }A|a_{11}=\rightarrow))=1-\frac{7}{8}(\frac{63}{64})^{n-2}.$$
In particular $$\lim\limits_{n\to\infty}P_{n}(A\text{
is solvable}|a_{11}=\rightarrow)=1.$$ 
Proving that $$\lim\limits_{n\to\infty}P_n(A\text{
is solvable}|a_{11}=\downarrow)=1$$ is the same as the latter case, where
now one defines $n-1$ "easy
winning games" $\{G'_i\}_{i=1}^{n}$ using the following (symmetric) picture:
\begin{center}\includegraphics[width=15cm]{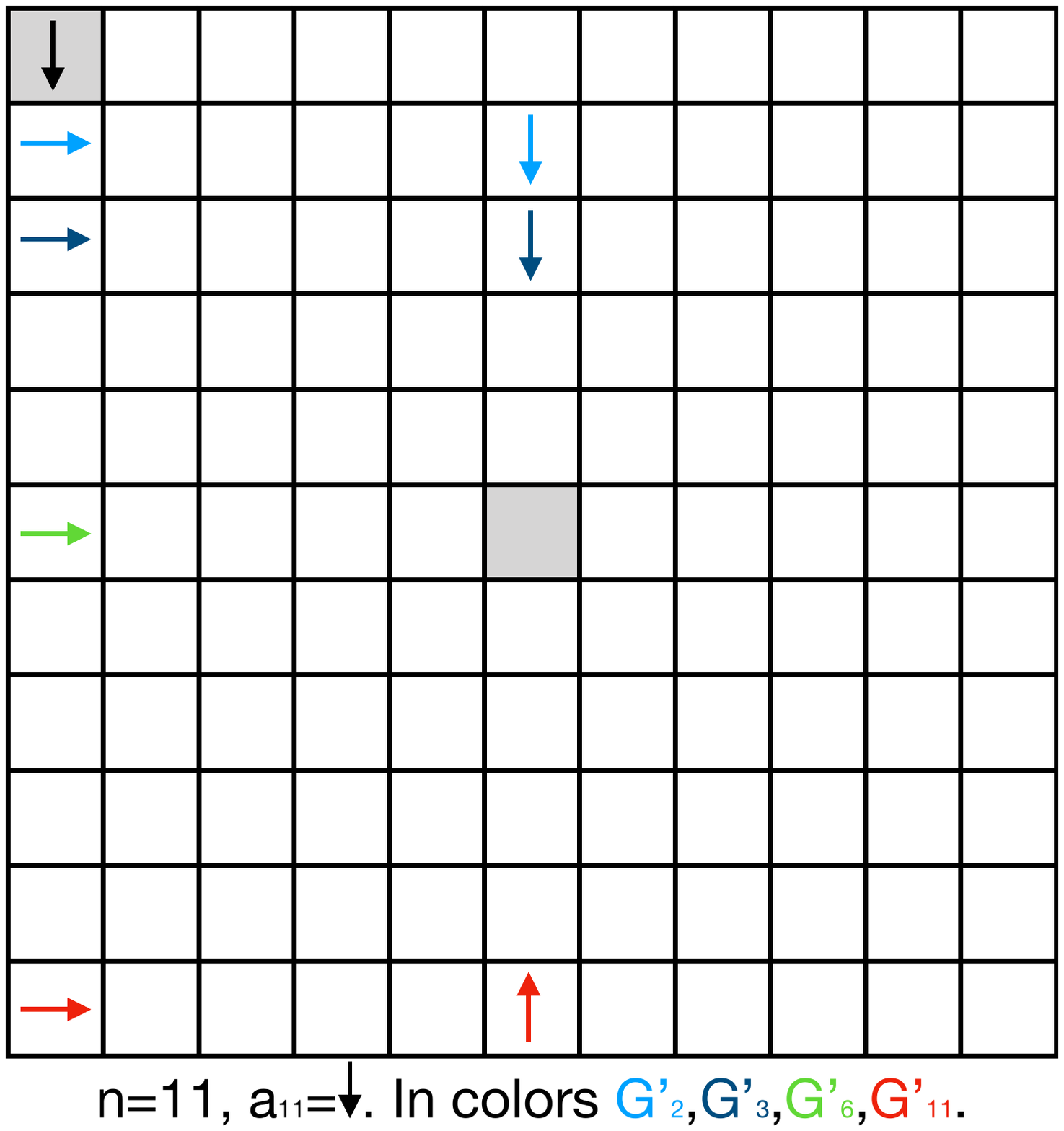}\end{center}
\end{proof}

\subsubsection{The exact asymptotic behaviour}\label{the_exact_asymptotic}Note that our proof of Proposition \ref{Prop_prob_solvable}
is quantitative and shows that $\frac{1}{3}-\text {Prob}_{A\in\text{Mat}_{n\times
n}(\mathcal{D})}(A\text{
is solvable})$,
which is the (strictly positive) error term in this limit, decays at least exponentially
with $n$. One can of course ask further questions regarding the exact asymptotic behaviour of $\lim\limits_{n\to\infty}\text {Prob}_{A\in\text{Mat}_{n\times
n}(\mathcal{D})}(A\text{
is solvable})$.
 
\subsection{The expected length of a random \emph{solvable} board}Let $n\geq3$
be an odd integer. We may consider the subset $Sol_n\subset\text{Mat}_{n\times
n}(\mathcal{D})$ consisting of all  solvable boards of size $n$. This is a finite set,
\ref{trivial_upper_bound_remark} may be read as $$\abs{Sol_{n}}< \frac{3}{8}\abs{\text{Mat}_{n\times
n}(\mathcal{D})}=\frac{3}{8}\cdot8^{n^2}=3\cdot8^{n^2-1},$$and Proposition
\ref{Prop_prob_solvable} may be read as $$\lim\limits_{n\to\infty}\abs{Sol_{n}}=\frac{3}{8}\abs{\text{Mat}_{n\times
n}(\mathcal{D})}=\frac{3}{8}\cdot8^{n^2}=3\cdot8^{n^2-1}.$$ We can now turn
$Sol_n$ into a probability space by assigning each board with equal probability
(the discrete uniform measure).
A random solvable board is a board chosen from $Sol_n$ according to this
probability measure. 
\subsubsection{Warning} The discrete uniform measure on  $\text{Mat}_{n\times
n}(\mathcal D)$ is
equivalent to assigning a direction to each entry in an $n\times n$ matrix
with probability $\frac{1}{8}$ to each direction (and the choices of each
two different entries are independent events); the uniform measure on $Sol_n$
does not have such an easy description. 

\

For the remaining of this section we consider the probability space of random
solvable boards. We ask $$E_n:=\mathbb{E}(\text{length }A)=?$$i.e., what is the expected
value of the length of a random solvable board of size $n$? We will show
that $\lim\limits_{n\to\infty}E_n=\frac{209}{96}$, i.e., the expected value
of the length of a big random solvable board tends to be slightly more than 2. Our proof below is quantitative and shows that $\abs{E_n-\frac{209}{96}}$,
which is the absolute value of error term in this limit, decays at least exponentially
with $n$ (cf. \ref{the_exact_asymptotic}).

\begin{lemma}[proved together with Noga Alon]\label{lemma_onetwothree_length_asymp_prob}For any odd integer
$n\geq3$:\begin{enumerate}\item ${Prob}_{A\in Sol_n}\{\text{length
}A=1\}>\frac{1}{3}$;\item ${Prob}_{A\in Sol_n}\{\text{length
}A=2\}>\frac{5}{32}$; \item ${Prob}_{A\in Sol_n}\{\text{length
}A=3\}>(1-(\frac{63}{64})^{n-2})\cdot\frac{49}{96}$; \item ${Prob}_{A\in Sol_n}\{\text{length
}A=k\}<(\frac{63}{64})^{n-2}\cdot\frac{49}{96}$ for any integer $k\geq4$.
\end{enumerate}
\end{lemma}

\begin{proof} 

Fix and odd integer $n\geq3$. 

\begin{enumerate}

\item The board $A$ has length 1 if and only if $a_{11}=\searrow$, so there
exist exactly $8^{n^2-1}$ solvable boards of length 1. Thus, $$\text
{Prob}_{A\in Sol_n}\{\text{length
}A=1\}=\frac{8^{n^2-1}}{\abs{Sol_n}}>\frac{8^{n^2-1}}{3\cdot8^{n^2-1}}=\frac{1}{3}.$$

\item The board $A$ has length 2 if and only if one of the following holds:

(i) $a_{11}=\rightarrow$ and either $a_{1,\frac{n+1}{2}}=\downarrow$ or $a_{1,n}=\swarrow$;

(ii) $a_{11}=\downarrow$ and either $a_{\frac{n+1}{2},1}=\rightarrow$ or
$a_{n,1}=\nearrow$.

We conclude that there exist exactly $2\cdot(1-\frac{7}{8}\cdot\frac{7}{8})\cdot8^{n^2-1}=2\cdot\frac{15}{64}\cdot8^{n^2-1}=30\cdot8^{n^2-3}$
solvable boards of length 2. Thus, $$\text
{Prob}_{A\in Sol_n}\{\text{length
}A=2\}=\frac{30\cdot8^{n^2-3}}{\abs{Sol_n}}>\frac{30\cdot8^{n^2-3}}{3\cdot8^{n^2-1}}=\frac{5}{32}.$$
 
\item Assume that $a_{11}=\rightarrow$ and retain the notion of the $n-1$
"easy winning games" $G_2,G_3,\dots ,G_n$ introduced in the proof of Proposition
\ref{Prop_prob_solvable}. A necessary condition for the board $A$ to have
length 3 is $a_{1,\frac{n+1}{2}}\neq\downarrow$ and $a_{1,n}\neq\swarrow$.
Note that if in addition $G_i$ is possible on the board $A$ for some $i\neq\frac{n+1}{2}$
then it is guaranteed that $A$ has length 3. Recall that for $i\neq \frac{n+1}{2}$ the probability of the game $G_i$ to
be possible on $A$ when $A$ is chosen by the uniform discrete probability
from $\{A\in\text{Mat}_{n\times n}(\mathcal D)|a_{11}=\rightarrow\}$ is $\frac{1}{64}$.
Moreover, the events $$\{G_i\text{ is possible on }A\text{ given }a_{11}=\rightarrow\}_{i\in\{1,2,\dots,\frac{n-1}{2},\frac{n+3}{2},\frac{n+5}{2},\dots,n\}}$$
are independent events. We conclude that there
exist at least $(\frac{7}{8})^2\cdot(1-(\frac{63}{64})^{n-2})\cdot 8^{n^2-1}$
boards $A$ of length 3 with $a_{11}=\rightarrow$.  A similar argument shows
that there
exist at least $(\frac{7}{8})^2\cdot(1-(\frac{63}{64})^{n-2})\cdot 8^{n^2-1}$
boards $A$ of length 3 with $a_{11}=\downarrow$. Altogether, we showed that$$\text
{Prob}_{A\in Sol_n}\{\text{length
}A=3\}\geq\frac{2\cdot (\frac{7}{8})^2\cdot(1-(\frac{63}{64})^{n-2})\cdot
8^{n^2-1}}{\abs{Sol_n}}$$ $$>\frac{2\cdot (\frac{7}{8})^2\cdot(1-(\frac{63}{64})^{n-2})\cdot
8^{n^2-1}}{3\cdot8^{n^2-1}}=(1-(\frac{63}{64})^{n-2})\cdot\frac{49}{96}.$$
\item It follows immediately from (1),(2),(3) and the facts that $\frac{1}{3}+\frac{5}{32}+\frac{49}{96}=1$
and that for any fixed $n$ we have $\sum\limits_{k=1}^{n^2}\text
{Prob}_{A\in Sol_n}\{\text{length
}A=k\}=1$. \end{enumerate}\end{proof}

\begin{lemma}\label{lemma_similar_to_monotone_convergence}For any $n\in2\mathbb{N}+1$
and any $m\in\mathbb{N}$ let $p_{nm}\in[0,1]$ be such that\begin{enumerate}
\item for any fixed $n$ we have that $p_{nk}=0$ for any $k>n^2$; 
\item for any fixed $n$ we have $\sum\limits_{m=1}^\infty p_{nm}\big(=\sum\limits_{m=1}^{n^2}
p_{nm}\big)=1$;
\item for any fixed $n$ we have $$p_{nm}\geq\begin{cases}
\frac{1}{3},\text{ }\text{ }\text{ }\text{ }\text{ }\text{ }\text{ }\text{
}\text{ }\text{ }\text{ }\text{ }\text{ }\text{ }\text{ }\text{ }\text{ }\text{
}\text{ }\text{ if }m=1\\
\frac{5}{32},\text{ }\text{ }\text{ }\text{ }\text{ }\text{ }\text{ }\text{
}\text{ }\text{ }\text{ }\text{ }\text{ }\text{ }\text{ }\text{ }\text{ }\text{
}\text{ if }m=2\\
(1-(\frac{63}{64})^{n-2})\cdot\frac{49}{96},\text{ if }m=3
\end{cases}.$$ 
\end{enumerate}
Then, $\lim\limits_{n\to\infty}(\sum\limits_{k=1}^{\infty}k\cdot
p_{nk})=\frac{209}{96}$.
\end{lemma}

\begin{proof}First note that $$\frac{1}{3}+2\cdot\frac{5}{32}+3\cdot\frac{49}{96}=\frac{32+30+147}{96}=\frac{209}{96}.$$
Fix an odd integer $n\geq3$. Note that (1)-(3) easily imply that $$0\leq p_{nm}\leq(\frac{63}{64})^{n-2}\cdot\frac{49}{64} \text{
for any } m\geq4,$$and  $$p_{nm}\leq\begin{cases}
\frac{1}{3}+(\frac{63}{64})^{n-2}\cdot\frac{49}{96},\text{ }\text{ }\text{
}\text{ }\text{ }\text{ }\text{ }\text{
}\text{ }\text{ }\text{ }\text{ }\text{ }\text{ }\text{ }\text{ }\text{ }\text{
}\text{ }\text{ if }m=1\\
\frac{5}{32}+(\frac{63}{64})^{n-2}\cdot\frac{49}{96},\text{ }\text{ }\text{
}\text{ }\text{ }\text{ }\text{ }\text{
}\text{ }\text{ }\text{ }\text{ }\text{ }\text{ }\text{ }\text{ }\text{ }\text{
}\text{ if }m=2\\
\frac{49}{96},\text{ }\text{ }\text{ }\text{ }\text{ }\text{ }\text{ }\text{
}\text{ }\text{ }\text{ }\text{ }\text{ }\text{ }\text{ }\text{ }\text{ }\text{
}\text{ }\text{ }\text{ }\text{ }\text{ }\text{ }\text{ }\text{ }\text{ }\text{
}\text{ }\text{ }\text{ }\text{ }\text{ }\text{ }\text{ }\text{ }\text{ }\text{
if }m=3
\end{cases}.$$

Thus, we have $$E_n:=\sum\limits_{k=1}^{\infty}k\cdot
p_{nk}=\sum\limits_{k=1}^{n^2}k\cdot
p_{nk}$$ $$\leq\big(\frac{1}{3}+(\frac{63}{64})^{n-2}\cdot\frac{49}{96}\big)+2\cdot\big(\frac{5}{32}+(\frac{63}{64})^{n-2}\cdot\frac{49}{96}\big)+3\cdot\frac{49}{64}+\sum\limits_{k=4}^{n^2}k\cdot
p_{nk}$$ $$\leq\big(\frac{1}{3}+(\frac{63}{64})^{n-2}\cdot\frac{49}{96}\big)+2\cdot\big(\frac{5}{32}+(\frac{63}{64})^{n-2}\cdot\frac{49}{96}\big)+3\cdot\frac{49}{64}+n^{4}\cdot(\frac{63}{64})^{n-2}\cdot\frac{49}{64}\xrightarrow[n\to\infty]{}\frac{209}{96},$$and
so $\limsup\limits_{n\to\infty}E_n\leq\frac{209}{96}$.

Similarly we have $$E_n:=\sum\limits_{k=1}^{\infty}k\cdot
p_{nk}=\sum\limits_{k=1}^{n^2}k\cdot
p_{nk}$$ $$\geq\frac{1}{3}+2\cdot\frac{5}{32}+3\cdot(1-(\frac{63}{64})^{n-2})\cdot\frac{49}{96}+\sum\limits_{k=4}^{n^2}k\cdot
p_{nk}$$ $$\geq\frac{1}{3}+2\cdot\frac{5}{32}+3\cdot(1-(\frac{63}{64})^{n-2})\cdot\frac{49}{96}\xrightarrow[n\to\infty]{}\frac{209}{96},$$and
so $\liminf\limits_{n\to\infty}E_n\geq\frac{209}{96}$.

We showed $$\frac{209}{96}\leq\liminf\limits_{n\to\infty}E_n\leq\limsup\limits_{n\to\infty}E_n\leq\frac{209}{96},$$
which implies $\lim\limits_{n\to\infty}E_n=\lim\limits_{n\to\infty}(\sum\limits_{k=1}^{\infty}k\cdot
p_{nk})=\frac{209}{96}$. \end{proof}

\begin{theorem}\label{theorem_on_expected_length}Let $E_n$ be the expected length of a random solvable board
of size $n$. Then, $$\lim\limits_{n\to\infty}E_n=\frac{209}{96}.$$ \end{theorem}

\begin{proof}It follows from Lemma \ref{lemma_onetwothree_length_asymp_prob} and Lemma \ref{lemma_similar_to_monotone_convergence} by setting $p_{nm}:=\text
{Prob}_{A\in Sol_n}\{\text{length
}A=m\}$. \end{proof}

\section{Longest possible board}\label{section_ML}For a fixed odd integer $n\geq3$ we may ask how "complicated" can a solvable board of size $n$ be, or more precisely what is the maximal length of a solvable board of size $n$. Denote this maximal length by $$ML(n):=\max\limits_{A\in Sol_n}\{\text{length }A\}.$$

\begin{example}\label{example_game_of_length_3n}For any odd $n\geq11$ there exists a board of size $n$ whose length is $3n$. Here is an explicit such board for $n=11$ (the blue arrows indicate the -- unique in this case -- shortest winning game):

\begin{center}\includegraphics[width=15cm]{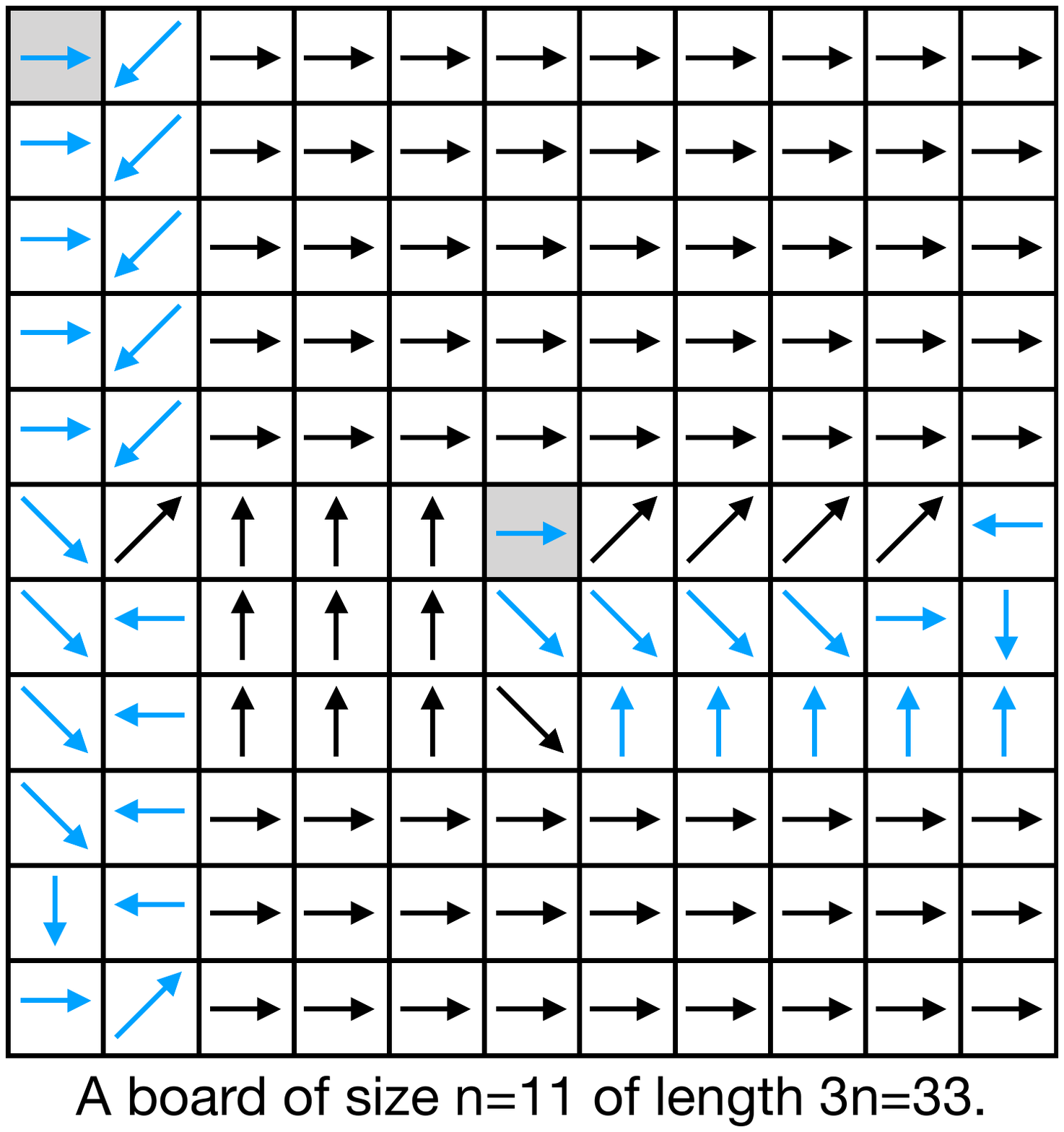}\end{center}
For $n\geq13$ we construct a similar board of size $3n$. E.g., for $n=13$ we take the $n=11$ board above, duplicate the third and the seventh columns, and then duplicate the second and the ninth rows. Similar construction works for any greater odd $n$ as well.

\end{example}

\subsection{Trivial bounds}Example \ref{example_game_of_length_3n} shows that $ML(n)\geq 3n$ for all $n\geq11$. Obviously the length of any solvable board of size $n$ is at most $n^2$, and so $ML(n)\leq n^2$. 

\begin{conjecture}\label{main_conjecture}
As $n\to\infty$, $ML(n)=O(n)$.
\end{conjecture}

The following is a weaker conjecture: 

\begin{conjecture}\label{weak_conjecture}
As $n\to\infty$, $ML(n)=o(n^{2})$.
\end{conjecture}

If conjecture \ref{main_conjecture} is true than one may ask what is the asymptotic behaviour exactly, e.g., $$c:=\inf\{C\in\bR_{>0}|\exists N\in\mathbb{N}\text{ s.t. }ML(n)\leq C\cdot n\text{ }\text{ }\forall n>N\}=?$$Example \ref{example_game_of_length_3n} shows that in this case $c\geq3$.

\section{The graph of a board and isomorphic boards}\label{section_graphs}

\begin{definition}\label{def_graph_of_board}Let $A\in\text{Mat}_{n\times n}(\mathcal{D})$ be a board of size $n$. We define $G(A)$, the directed "double rooted" graph associated to $A$: $G(A)$ is the directed graph with $n^2$ vertices labeled by $$(i,j)\in\{1,\dots,n\}\times\{1,\dots,n\}$$ (the positions of the matrix
$A$), whose "beginning root" is $(1,1)$, whose "end root" is $(\frac{n+1}{2},\frac{n+1}{2})$, and with an edge pointing from $(i,j)$ to $(i',j')$ if and only if $a_{ij}$ is directing to the position $(i',j')$. 

Then, a game on the board $A$ corresponds to a path on the directed graph that starts in the beginning root, and it is a winning game if and only if this path ends at the end root. The length of a board $A$ is the shortest path from $(1,1)$ to $(\frac{n+1}{2},\frac{n+1}{2})$ on $G(A)$, where the length of a path is defined to be the number of edges it contains (counted with multiplicities). 
\end{definition}

\begin{example}Board 4, its associated graph and two games on them are shown below:
\begin{center}\includegraphics[width=14cm]{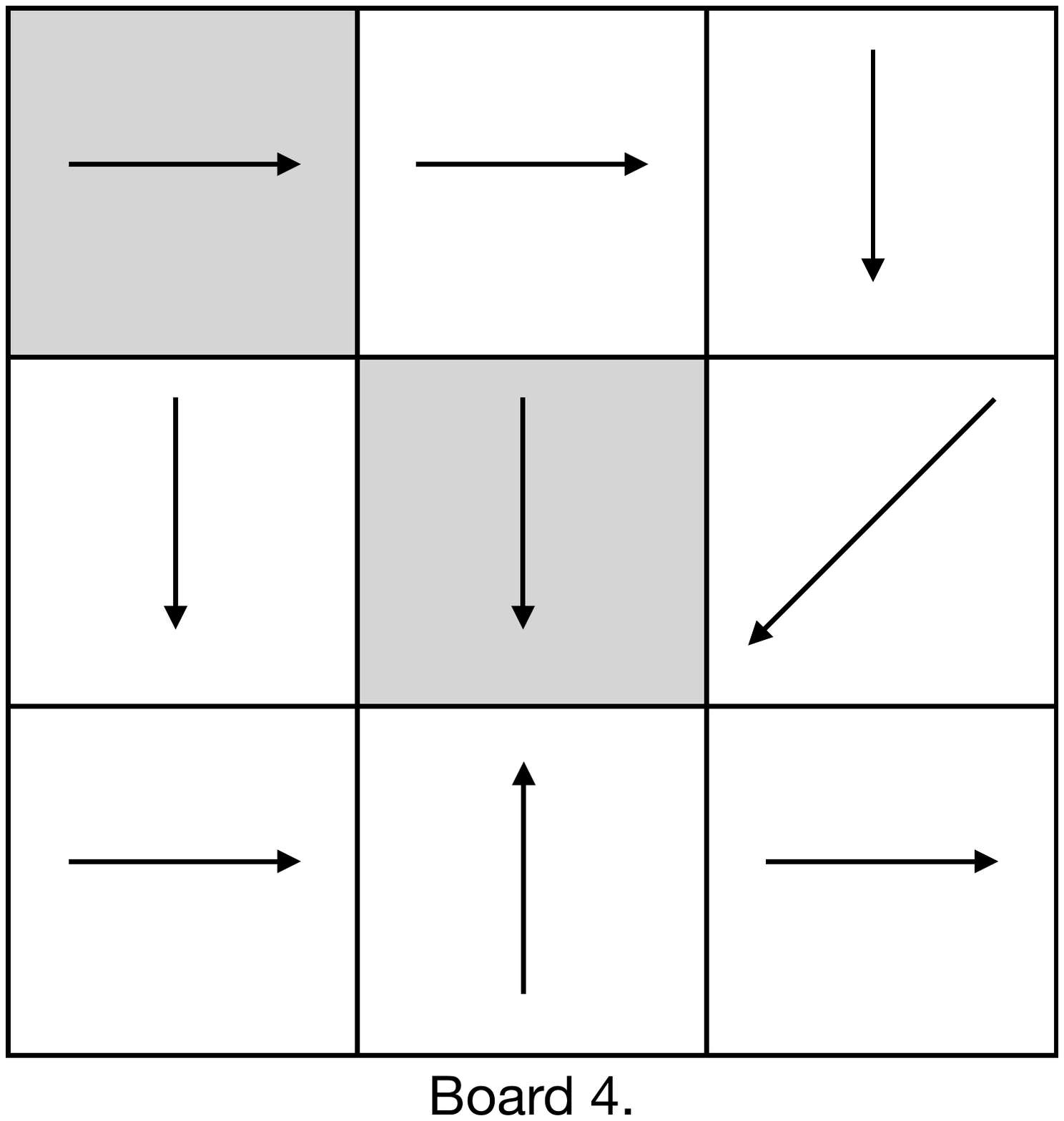}\end{center}
\begin{center}\includegraphics[width=14cm]{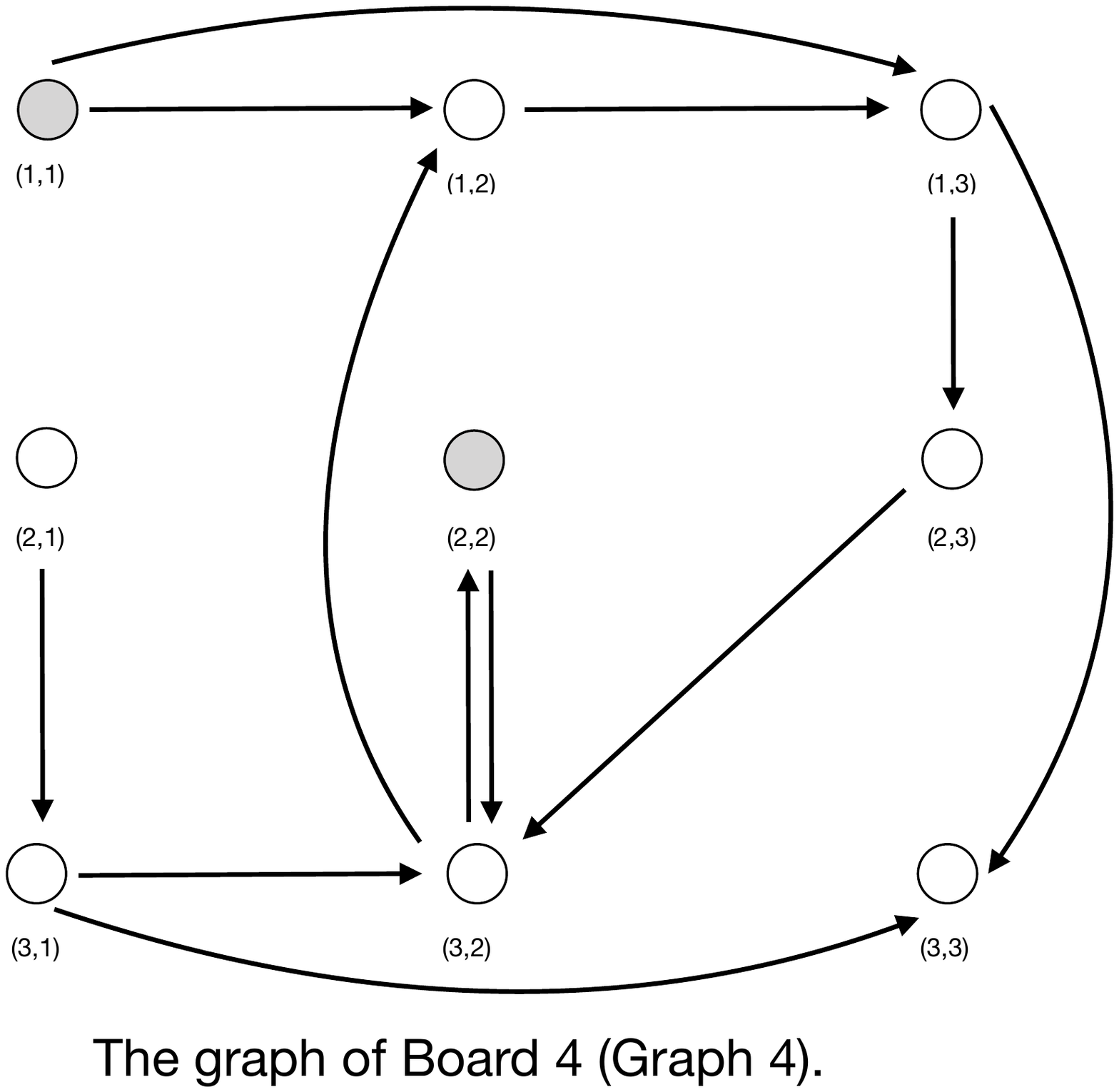}\end{center}
\begin{center}\includegraphics[width=14cm]{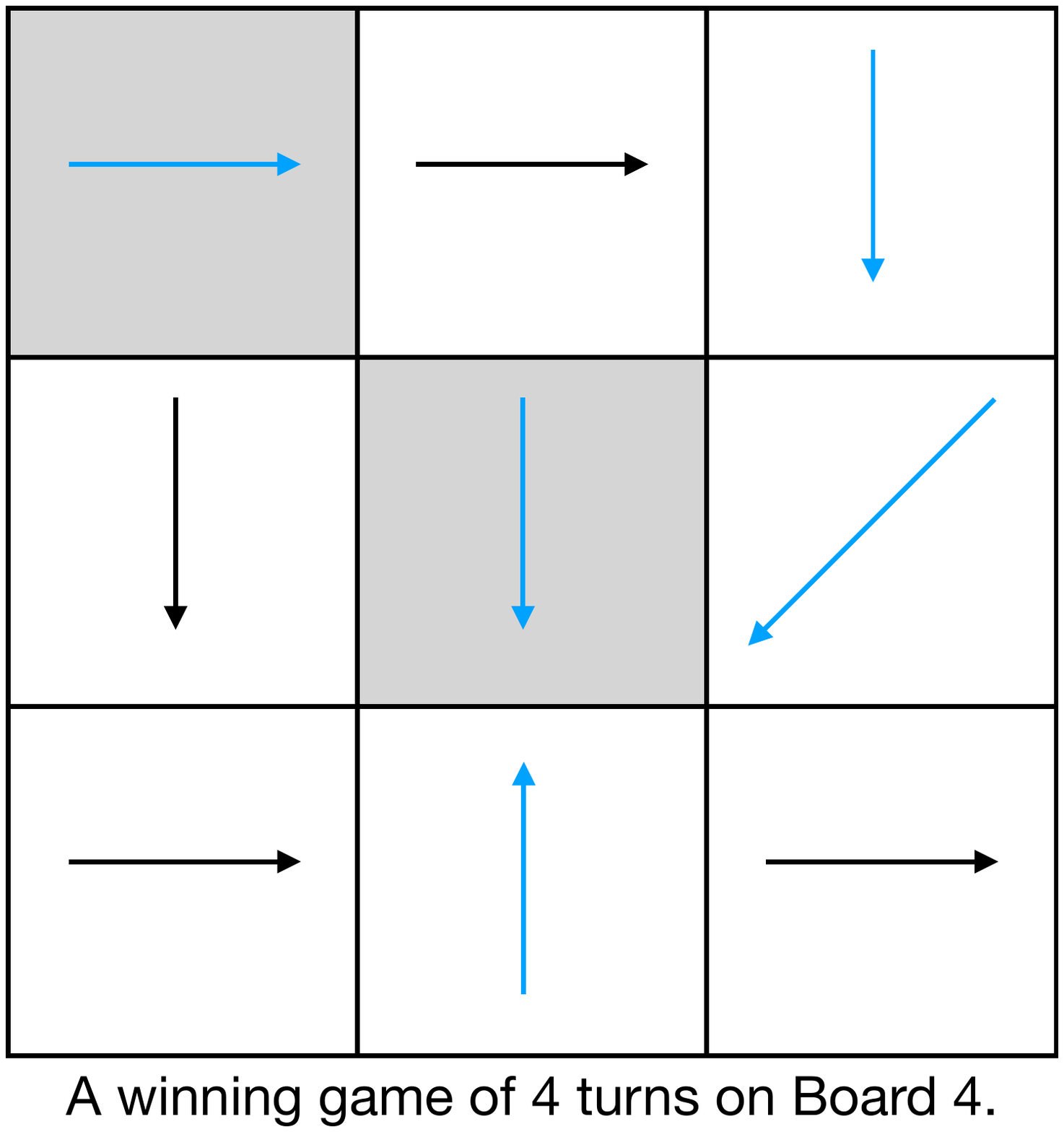}\end{center}
\begin{center}\includegraphics[width=14cm]{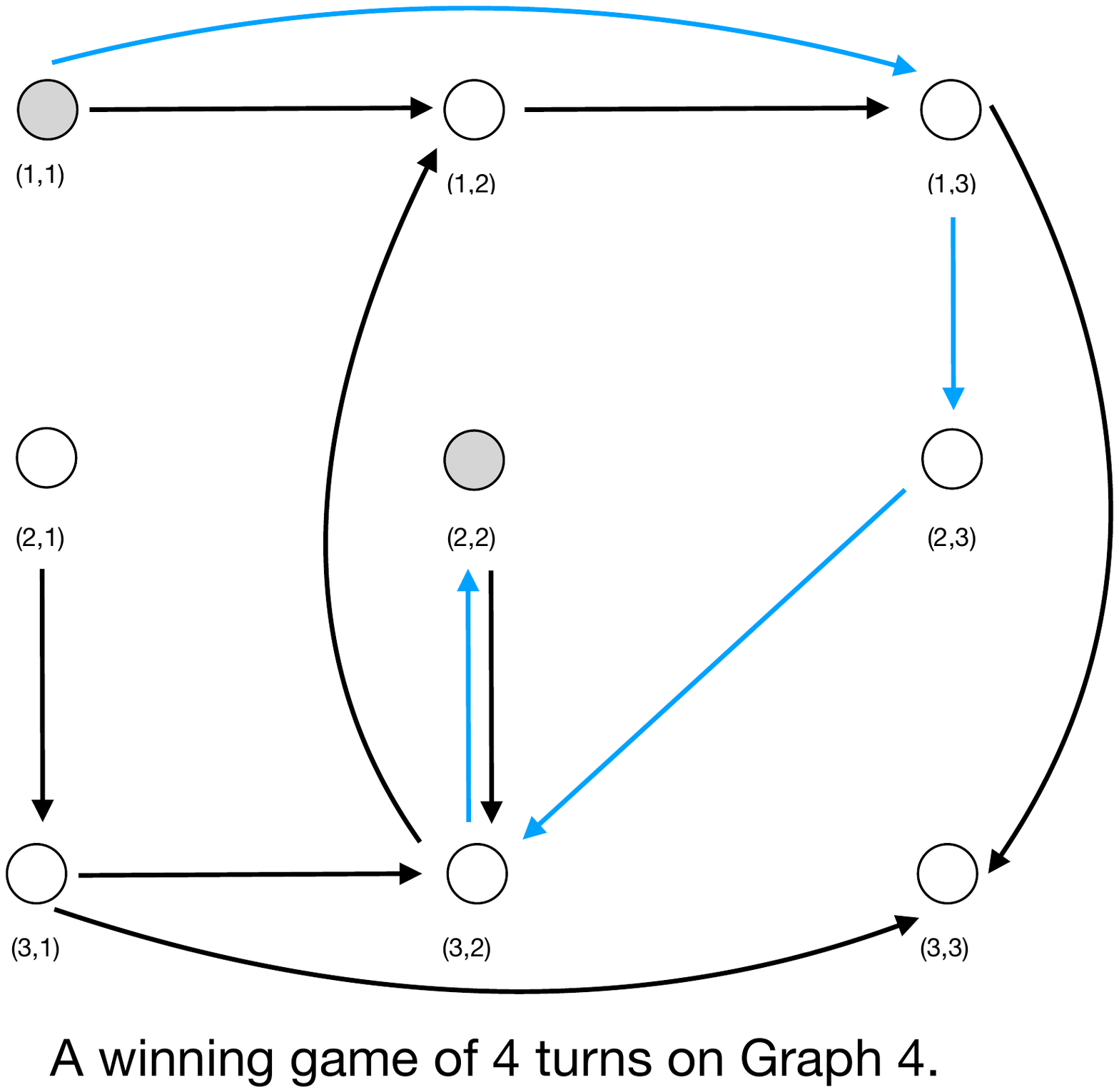}\end{center}
\begin{center}\includegraphics[width=14cm]{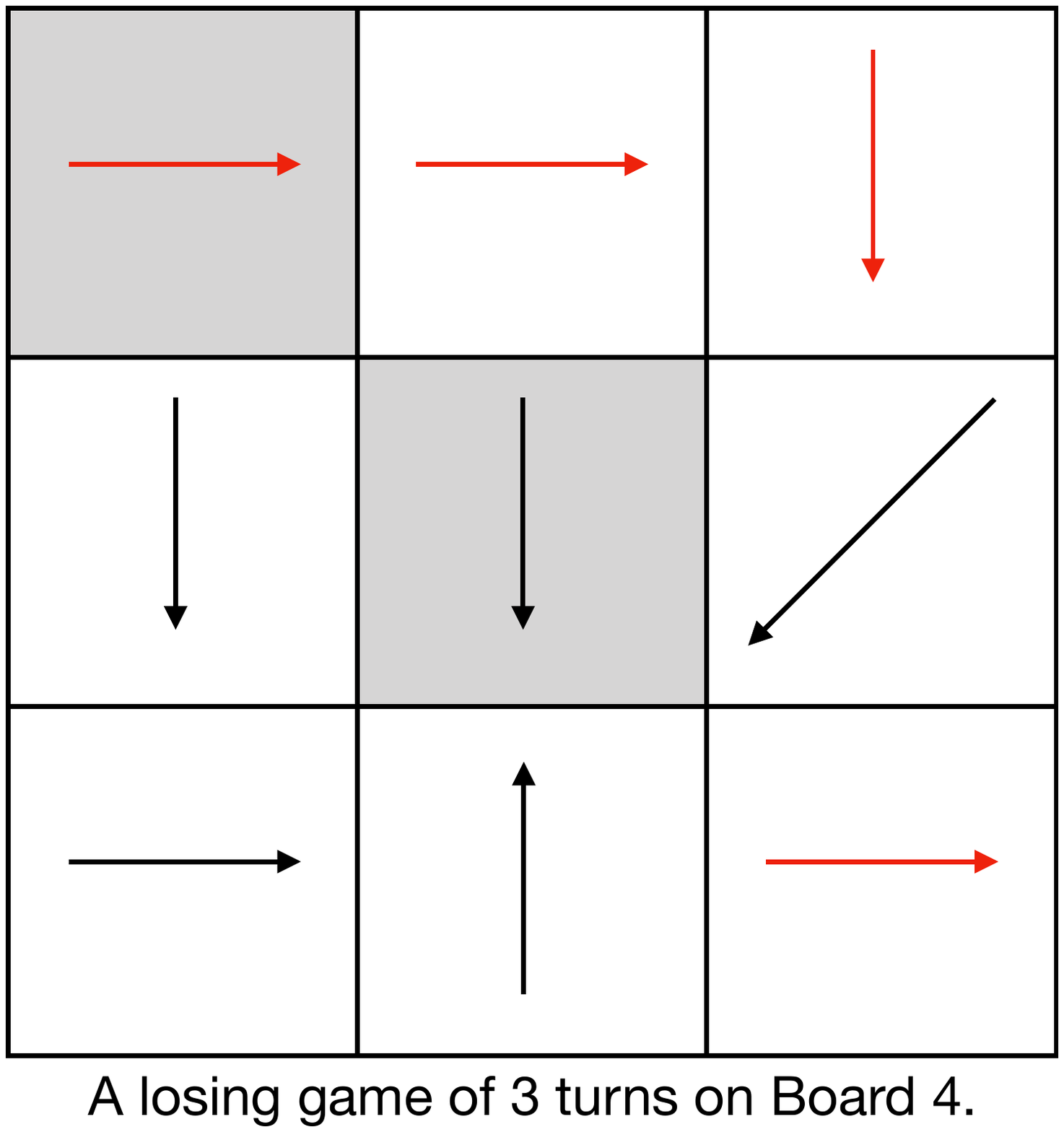}\end{center}
\begin{center}\includegraphics[width=14cm]{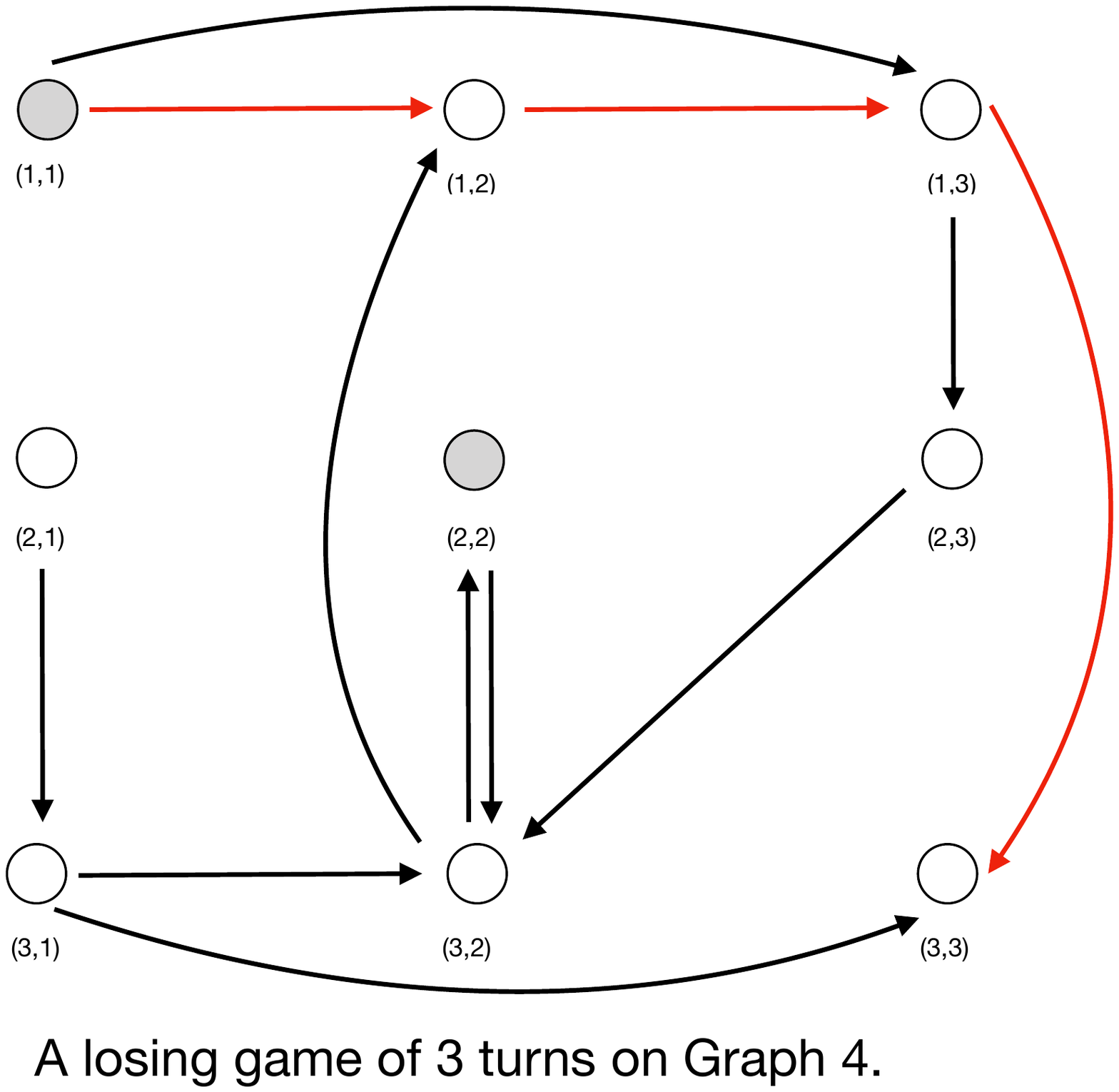}\end{center}
\end{example}


\begin{proposition}\label{prop-total-number-of-edges}Let $A\in\text{Mat}_{n\times n}(\mathcal{D})$ be a board, and let $\Gamma=G(A)$ be its graph. Denote the $(1,1)$ vertex by $v_{\text{beginning}}$ and the $(\frac{n+1}{2},\frac{n+1}{2})$ vertex by $v_{\text{end}}$. Finally, for a vertex $v\in\Gamma$ denote by $Out(v)$ the number of edges in $\Gamma$ pointing from $v$ (also sometimes denoted $\text{deg}^+(v)$). Then:
\begin{enumerate}
\item For any vertex $v\in\Gamma$, $0\leq Out(v)\leq n-1$. Moreover, $Out(v_{beginning})\in\{0,n-1\}$ and $Out(v_{end})=\frac{n-1}{2}$.
\item $\sum\limits_{v\in\Gamma}Out(v)$, i.e., the total number of edges in $\Gamma$, satisfies $$\frac{n^3}{6}-\frac{n^2}{2}+\frac{5n}{6}-\frac{1}{2}\leq\sum\limits_{v\in\Gamma}Out(v)\leq\ \frac{5n^3}{6}-\frac{n^2}{2}+\frac{n}{6}-\frac{1}{2},$$ and these bounds are sharp, i.e., there exist boards $A',A''\in\text{Mat}_{n\times n}(\mathcal{D})$ with $$\sum\limits_{v\in G(A')}Out(v)=\frac{5n^3}{6}-\frac{n^2}{2}+\frac{n}{6}-\frac{1}{2},$$ $$\sum\limits_{v\in G(A'')}Out(v)=\frac{n^3}{6}-\frac{n^2}{2}+\frac{5n}{6}-\frac{1}{2}.$$
\end{enumerate}
\end{proposition}
\begin{proof}
(1) is obvious. To prove (2), for $v:=(i,j)\in\{1,\dots,n\}^2$ we define "the distance of $(i,j)$ to the center" by $d(v):=\max\{\abs{i-\frac{n+1}{2}},\abs{j-\frac{n+1}{2}}\}$ (see illustration
below). Then, for any vertex $v\in\Gamma$ we have $$0\leq d(v)\leq\frac{n-1}{2},$$and $$Out((i,j))\in\{\frac{n-1}{2}+d((i,j)),\frac{n-1}{2}-d((i,j))\}.$$ Moreover, there exists exactly 1 vertex $v\in\Gamma$ with $d(v)=0$, exactly 8 vertices $v\in\Gamma$ with $d(v)=1$ and in general for $1\leq k\leq\frac{n-1}{2}$ there exist exactly $4(2k+1)-4=8k$ vertices $v\in\Gamma$ with with $d(v)=k$.
\begin{center}\includegraphics[width=15cm]{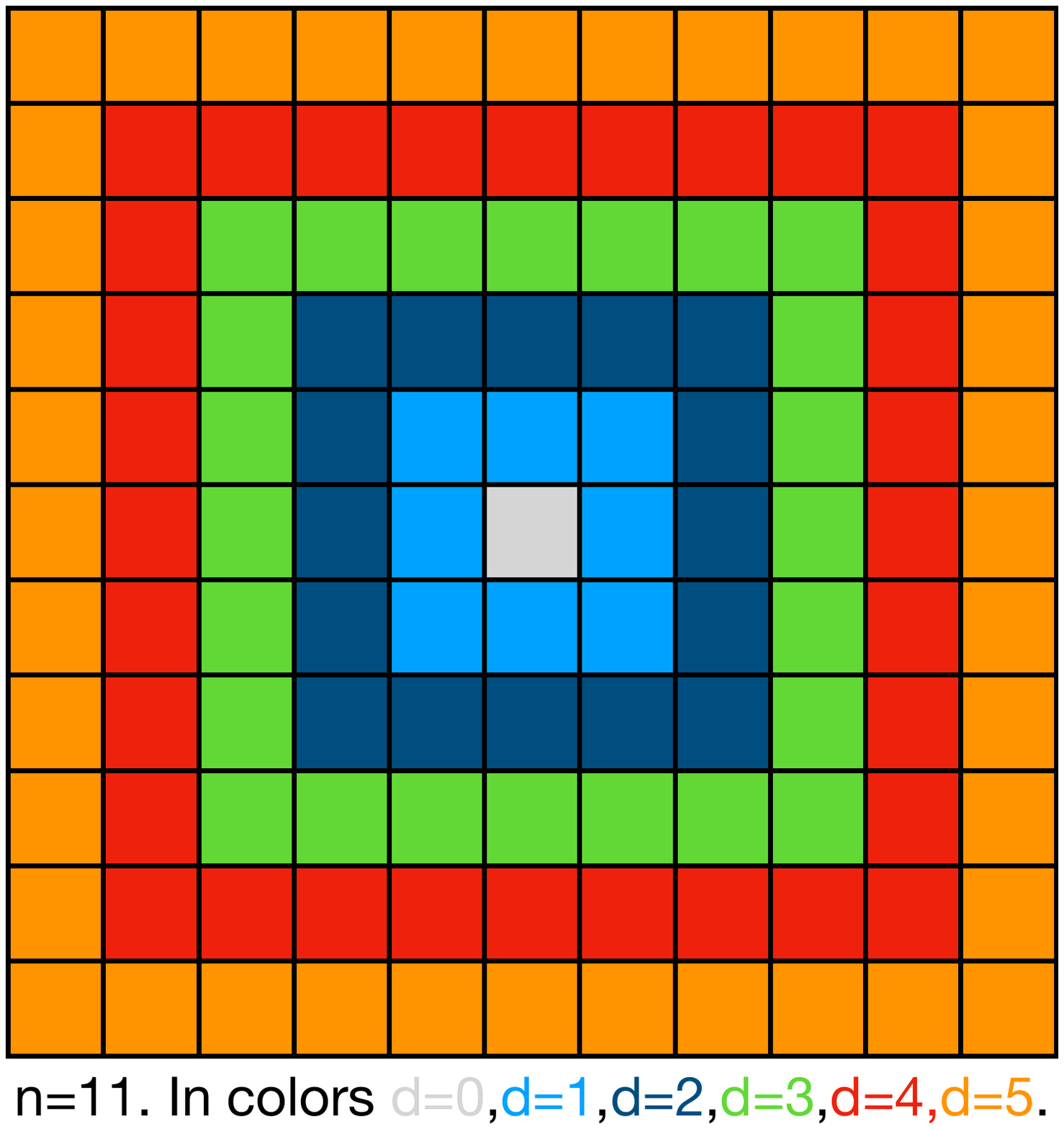}\end{center}
Now we can explicitly calculate $\sum\limits_{v\in\Gamma}Out(v)$ by summing over $d(v)$. To obtain an upper bound we use $Out(v)\leq\frac{n-1}{2}+d(v)$, and to obtain a lower bound we use $Out(v)\geq\frac{n-1}{2}-d(v)$.
 Let us start with an upper bound:
$$\sum\limits_{v\in\Gamma}Out(v)=\sum\limits_{k=0}^{\frac{n-1}{2}}\sum\limits_{v\in\Gamma,d(v)=k}Out(v)$$ $$\leq\sum\limits_{k=0}^{\frac{n-1}{2}}\sum\limits_{v\in\Gamma,d(v)=k}(\frac{n-1}{2}+k)$$ $$=(n-1)+\sum\limits_{k=1}^{\frac{n-1}{2}}\sum\limits_{v\in\Gamma,d(v)=k}(\frac{n-1}{2}+k)$$ $$=(n-1)+\sum\limits_{k=1}^{\frac{n-1}{2}}8k\cdot(\frac{n-1}{2}+k)$$ $$=(n-1)+\frac{4}{3}\cdot \frac{n-1}{2}\cdot \frac{n+1}{2}\cdot \frac{5n-3}{2}$$ $$=(n-1)+\frac{(n-1)(n+1)(5n-3)}{6}$$ $$=\frac{(n^2-1)(5n-3)}{6}+n-1=\frac{5n^3}{6}-\frac{n^2}{2}+\frac{n}{6}-\frac{1}{2}.$$
For the lower bound we have a very similar calculation:
$$\sum\limits_{v\in\Gamma}Out(v)=\sum\limits_{k=0}^{\frac{n-1}{2}}\sum\limits_{v\in\Gamma,d(v)=k}Out(v)$$
$$\geq\sum\limits_{k=0}^{\frac{n-1}{2}}\sum\limits_{v\in\Gamma,d(v)=k}(\frac{n-1}{2}-k)$$
$$=(n-1)+\sum\limits_{k=1}^{\frac{n-1}{2}}\sum\limits_{v\in\Gamma,d(v)=k}(\frac{n-1}{2}-k)$$
$$=(n-1)+\sum\limits_{k=1}^{\frac{n-1}{2}}8k\cdot(\frac{n-1}{2}-k)$$ $$=(n-1)+\frac{4}{3}\cdot
\frac{n-3}{2}\cdot\frac{n-1}{2}\cdot \frac{n+1}{2}$$ $$=(n-1)+\frac{(n-1)(n+1)(n-3)}{6}$$
$$=\frac{(n^2-1)(n-3)}{6}+n-1=\frac{n^3}{6}-\frac{n^2}{2}+\frac{5n}{6}-\frac{1}{2}.$$
It is easy to deduce the sharpness of these bounds from this proof: take $A'$ to be any board with $Out(v)=\frac{n+1}{2}+d(v)$ for any $v\in G(A)$, where clearly many such boards exist (and similarly construct $A''$).\end{proof}

\begin{remark}\label{remark_other_properties_of_graphs}One can formulate another necessary condition in the following spirit: retain the notation of Proposition \ref{prop-total-number-of-edges} and in addition for a vertex $v\in\Gamma$ denote by $In(v)$ the number of
edges in $\Gamma$ pointing to $v$ (also sometimes
denoted $\text{deg}^-(v)$). Then, if for a given vertex $v\in\Gamma$ we have that $In(v)$ is "very big", then $Out(v)$ is "quite small". For example, if $In(v)=4(n-1)$, then $d(v)=0$ and $Out(v)=\frac{n-1}{2}$; if $In(v)=4(n-1)-2$, then $0\leq d(v)\leq1$ and $Out(v)\leq \frac{n+1}{2}$, and so on -- this is because only entries in the matrix that are "very close to the center" lie in two "long" diagonals, and so may possibly have "many" arrows pointing to them (see illustrations below).
\begin{center}\includegraphics[width=15cm]{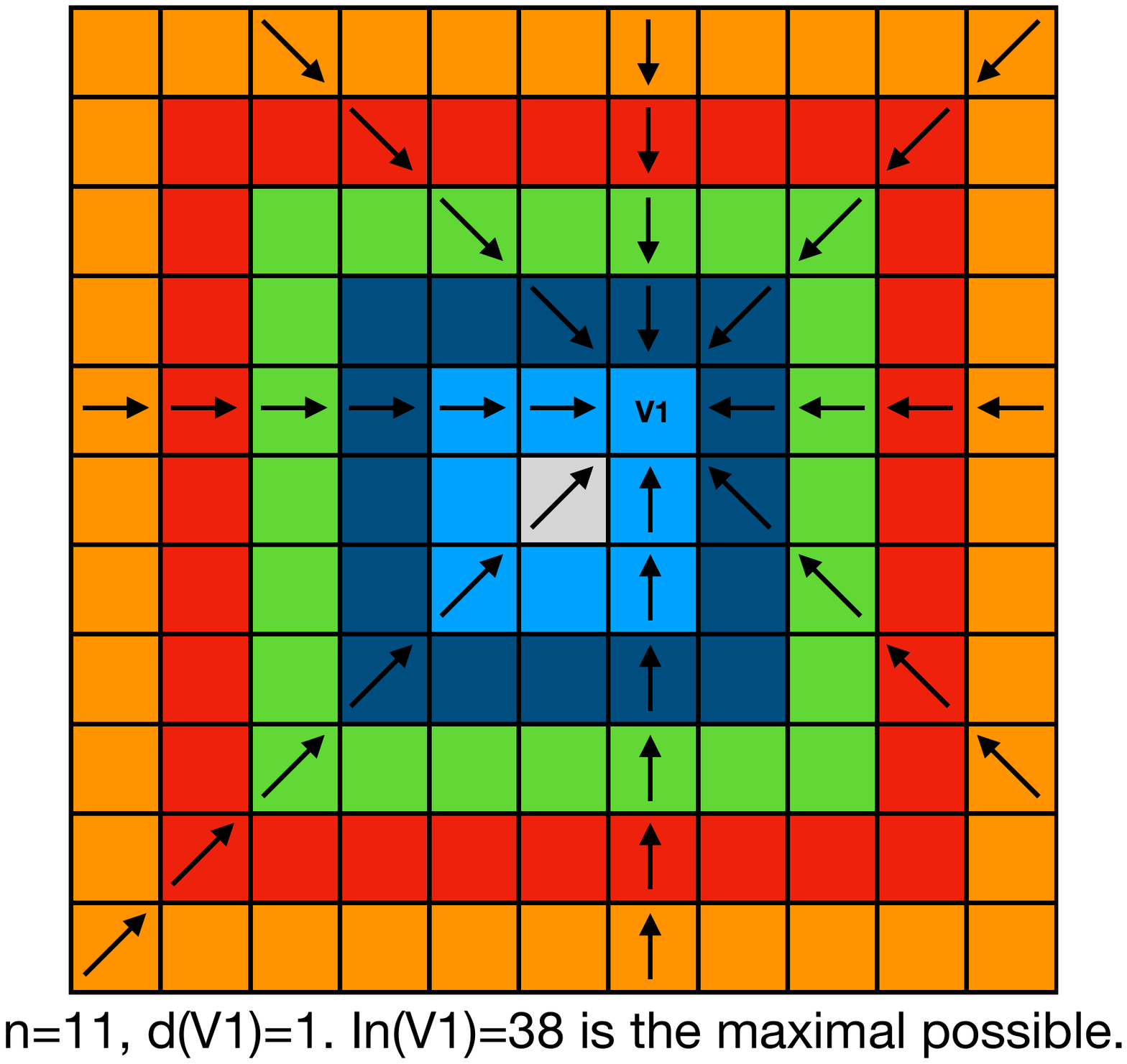}\end{center}
\begin{center}\includegraphics[width=15cm]{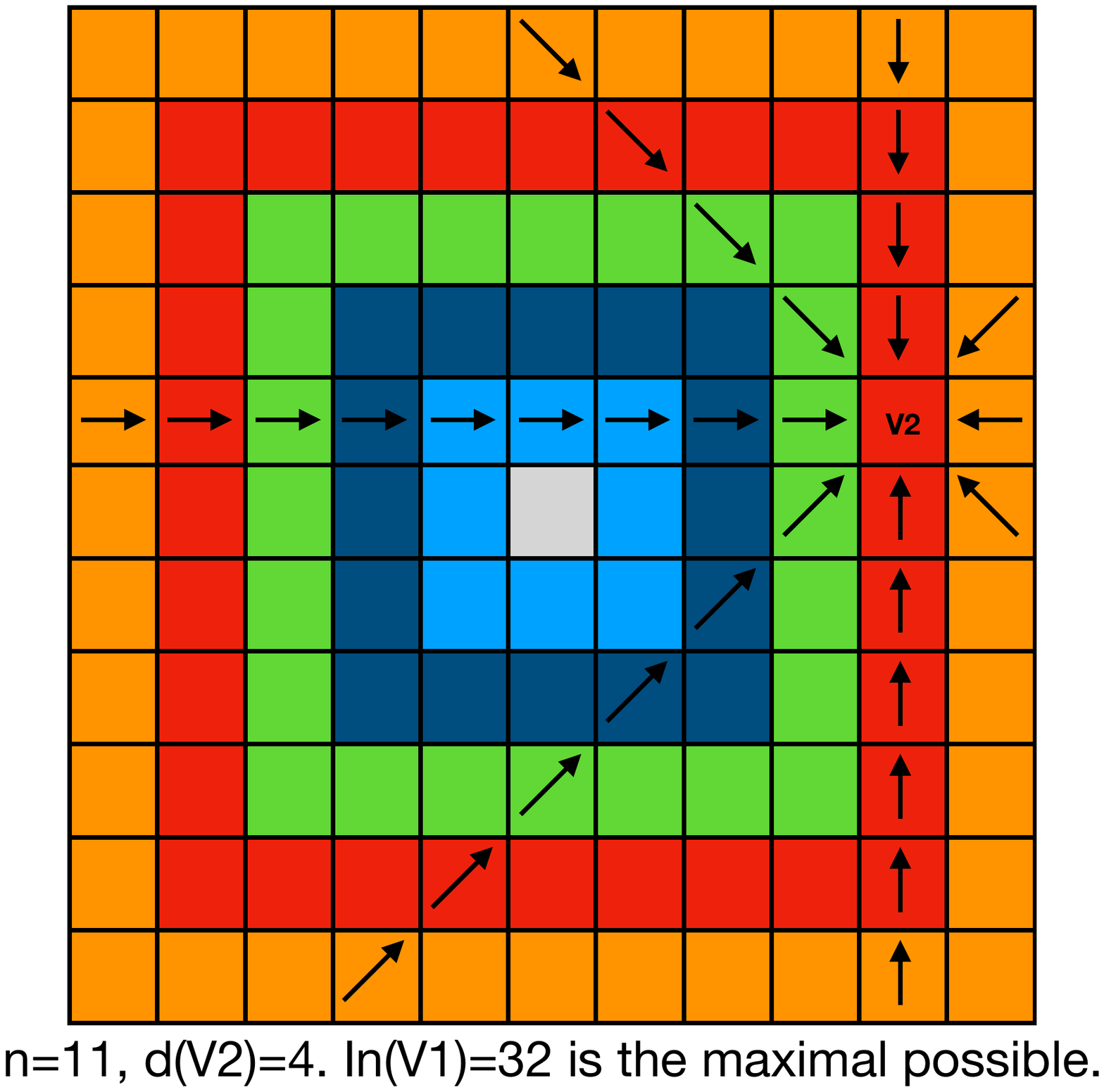}\end{center}
\end{remark}

\subsection{Isomorphic boards and isomorphic graphs}Let $G,G'$ be two double rooted directed graphs, and let $V,V'$ be the sets of vertices of $G,G'$ (respectively). We say that a bijection $\phi:V\to V'$ is an isomorphism of $G$ and $G'$ if $\phi$ maps the beginning root of $G$ to the beginning root of $G'$, the end root of $G$ to the end root of $G'$, and moreover for any $v_1,v_2\in V$ there exists an edge in $G$ pointing from $v_1$ to $v_2$ if and only if there exists an edge in $G'$ pointing
from $\phi(v_1)$ to $\phi(v_2)$. As usual if such a $\phi$ exists we say the $G$ and $G'$ are isomorphic and denote $G\cong G'$.

We say that the boards $A$ and $B$ are isomorphic (and denote $A\cong B$) if $G(A)\cong G(B)$. Note that if $A\cong B$, then $A$ is solvable if and only if $B$ is solvable, and moreover if $A$ is solvable then the length of $A$ equals the length of $B$.

\

Define an involution $(\cdot)^T:\mathcal{D}\to\mathcal{D}$ by $$\uparrow^T=\leftarrow$$ $$\nearrow^T=\swarrow$$ $$\rightarrow^T=\downarrow$$ $$\searrow^T=\searrow$$ $$\downarrow^T=\rightarrow$$ $$\swarrow^T=\nearrow$$ $$\leftarrow^T=\uparrow$$ $$\nwarrow^T=\nwarrow$$
This may be thought of as the involution "reflecting directions through the NW-SE line". 

Define an involution $R:\text{Mat}_{n\times n}(\mathcal{D})\to \text{Mat}_{n\times n}(\mathcal{D})$ by $R(A)_{ij}=(a_{ji})^T$. We think of $R$ as being the map that takes a board and returns the board obtained by putting a mirror on the main diagonal of $A$, or as "a physical reflection through the main diagonal". Note that usually $R(A)\neq A^T$, where $A^T$ is the standard transpose: $(A^T)_{ij}=a_{ji}$. 

For a set $X$ denote by $\text{Per}(X)$ the group of all bijective functions from $X$ to $X$ (permutations). Let $\sigma\in\text{Per}(\{1,\dots,n\}^2)$ and $\tau\in\text{Per}(\mathcal{D})$. We define a permutation $S^\sigma_\tau\in\text{Per}(\text{Mat}_{n\times n}(\mathcal{D}))$ by $S^\sigma_\tau(A)_{ij}:=\tau(a_{\sigma{(ij)}})$. Note that this is indeed a permutation: $(S^\sigma_\tau)^{-1}=S^{\sigma^{-1}}_{\tau^{-1}}$. For instance, if $\sigma'(ij)=ji$ and $\tau'(\cdot)=(\cdot)^T$, then $R=S^{\sigma'}_{\tau'}$. 

Denote by $\text{Sym}(\text{Mat}_{n\times n}(\mathcal{D}))$ the group $$\{S^\sigma_\tau\in\text{Per}(\text{Mat}_{n\times n}(\mathcal{D}))|\sigma\in\text{Pet}(\{1,\dots,n\}^2),\tau\in\text{Per}(\mathcal{D}), \forall A\in\text{Mat}_{n\times n}(\mathcal{D}):S^\sigma_\tau(A)\cong A\}.$$Obviously, $A\cong R(A)=S^{\sigma'}_{\tau'}(A)$ for any $A\in\text{Mat}_{n\times n}(\mathcal{D})$ -- the explicit graph isomorphism sends the vertex $(i,j)\in G(A)$ to the vertex $(j,i)\in G(R(A))$, i.e., $R\in\text{Sym}(\text{Mat}_{n\times n}(\mathcal{D}))$.  
\begin{conjecture}\label{conj_on_symmetries}$\text{Sym}(\text{Mat}_{n\times n}(\mathcal{D}))=\{Id,R\}$. \end{conjecture}

Conjecture \ref{conj_on_symmetries} most likely holds, however showing it is very technical, and moreover does not exhaust all isomorphism classes of graphs, as illustrated by the following "trivial change": the boards $A=\begin{pmatrix}\to & \to & \to \\ \to & \to & \to \\ \to & \to & \to \end{pmatrix}$ and $B=\begin{pmatrix}\to & \to & \uparrow \\ \to
& \to & \to \\ \to & \to & \to \end{pmatrix}$ are isomorphic, however clearly there does not exist $S\in\text{Sym}(\text{Mat}_{3\times 3}(\mathcal{D}))$ such that $S(A)=B$. 
\begin{question}\label{question_on_isomorphic_boards}Let $A,B\in \text{Mat}_{n\times n}(\mathcal{D})$ be such that $A\cong B$. Is it always the case that $B$ can be obtained from $A$ by applying a finite sequence of operations from $\text{Sym}(\text{Mat}_{n\times n}(\mathcal{D}))$ and "trivial changes"?\end{question}

Proposition \ref{prop-total-number-of-edges}, Remark \ref{remark_other_properties_of_graphs}, Conjecture \ref{conj_on_symmetries} and Question \ref{question_on_isomorphic_boards} are all strongly related to the following, much more general, question:

\begin{question}\label{general_question_on_how_many_graphs_are_there}Fixing an odd integer $n\geq3$ there are $8^{n^2}$ different boards of size $n$. (i) How many different boards are there up to isomorphism, i.e., how many different graphs of the form $G(A)$ are there? (ii) Given a double rooted graph with $n^2$ vertices, how can we tell whether this graph is of the form $G(A)$ for some board $A$ of size $n$? \end{question}

\section{Some generalizations and further remarks}\label{section_remarks}

\subsection{Determining the solvability and length of a given board}Given
a board $A$ we would like to know whether it is solvable and if it is what
is its length. As explained in Definition \ref{def_graph_of_board} this is
equivalent to determining whether there exists a path from $(1,1)$ to $(\frac{n+1}{2},\frac{n+1}{2})$
on $G(A)$ and if such a path exits what is the length of a shortest such path. Such an
algorithm is well known and is usually called the Breadth-First Search (BFS)
algorithm (see \cite[13.3,13.4]{BFS}). It is known to be linear in the number of edges
of the graph, which is $O(n^3)$ is our case, by Proposition \ref{prop-total-number-of-edges}(2).

Here is an intuitive explanation: how do
we know that
Board 2 introduced  in Example \ref{first_example} is not solvable? Well,
no position in this board is directing
to the $(3,3)$ position. This suggests the following algorithm to determine
whether a given board $A$ is solvable or not: mark all positions that point
to the $(\frac{n+1}{2},\frac{n+1}{2})$ position,
then mark  all positions that point to already marked positions and so on,
until you mark -- or do not mark -- the $(1,1)$ position. The BFS algorithm
is based on this idea, where there one first mark all points that $(1,1)$
is pointing to, then mark all points that already marked points are pointing
to, and so on.

\subsection{Generalized games on boards}For a given board our games
always start at the $(1,1)$ position and the goal is always to arrive at
the $(\frac{n+1}{2},\frac{n+1}{2})$ position. For a fixed board $A$ one may
ask whether it is possible to arrive from any fixed position to any fixed
position, whether it is possible to arrive from all positions to all positions,
and other similar questions. We may also ask what is the probability of a
random board to have these properties. 

For a fixed board $A$ such questions may
be formulated on the associated graph $G(A)$, for instance: given
 any directed graph one may contract any strongly connected component of it
(a set of vertices that are all reachable from one to another) into a single
vertex, and the resulting graph will be a Directed Acyclic Graph (DAG), i.e.,
a directed graph with no cycles. Then, the question whether it is possible
to arrive from all positions to all positions on the board $A$ is equivalent
to whether the DAG obtained from $G(A)$ is a point.

\subsection{Finite field model}(following is an idea by
Michele Fornea).

\

The wind compass may be realized quite naturally as $\mathbb{F}_9^{\times}\cong((\mathbb{Z}/3\mathbb{Z})[x]/(x^2+1))^{\times}$;
here is the explicit identification:

\begin{center}\includegraphics[width=17cm]{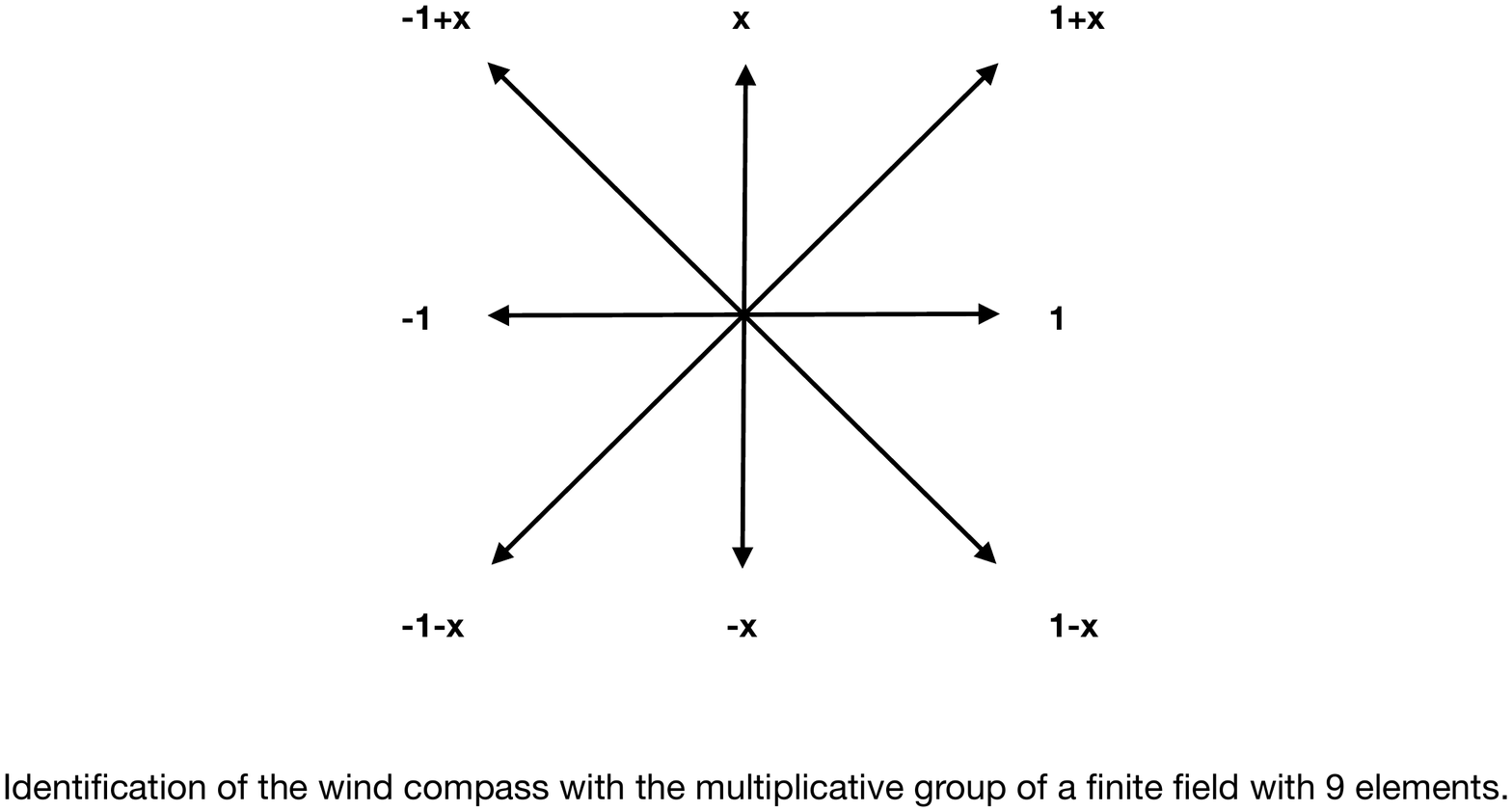}\end{center}

Then, any board may be thought of as an element of $\text{Mat}_{n\times n}(\mathbb{F}_9^{\times})\subset
\text{Mat}_{n\times n}(\mathbb{F}_9)$. If we allow generalized boards in
which some entries may be zero (in addition to the 8 winds) we may define
all the usual matrix operations on boards. One can say, for instance, that
a position that has a zero in it is pointing nowhere, and so once a player
arrived at such a position the game is over and unless it is the $(\frac{n+1}{2},\frac{n+1}{2})$
position it is a loosing game. We can now ask many questions relating matrix
theory to boards, for instance: let $A,B\subset\text{Mat}_{n\times n}(\mathbb{F}_9)$
be two boards of the same size. Can we say anything about the graphs $G(AB)$ or $G(A+B)$ if we know something about the graphs $G(A)$ and $G(B)$?

\subsection{Torus boards} Let $A$ be a board of size $n$. We may place $A$ on a torus in the usual way, which is equivalent to tiling $\mathbb{Z}^2$ with infinitely many copies of $A$ and identifying any position $(i,j)$ with the position $(i+n,j)$ and with the position $(i,j+n)$.

\begin{center}\includegraphics[width=17cm]{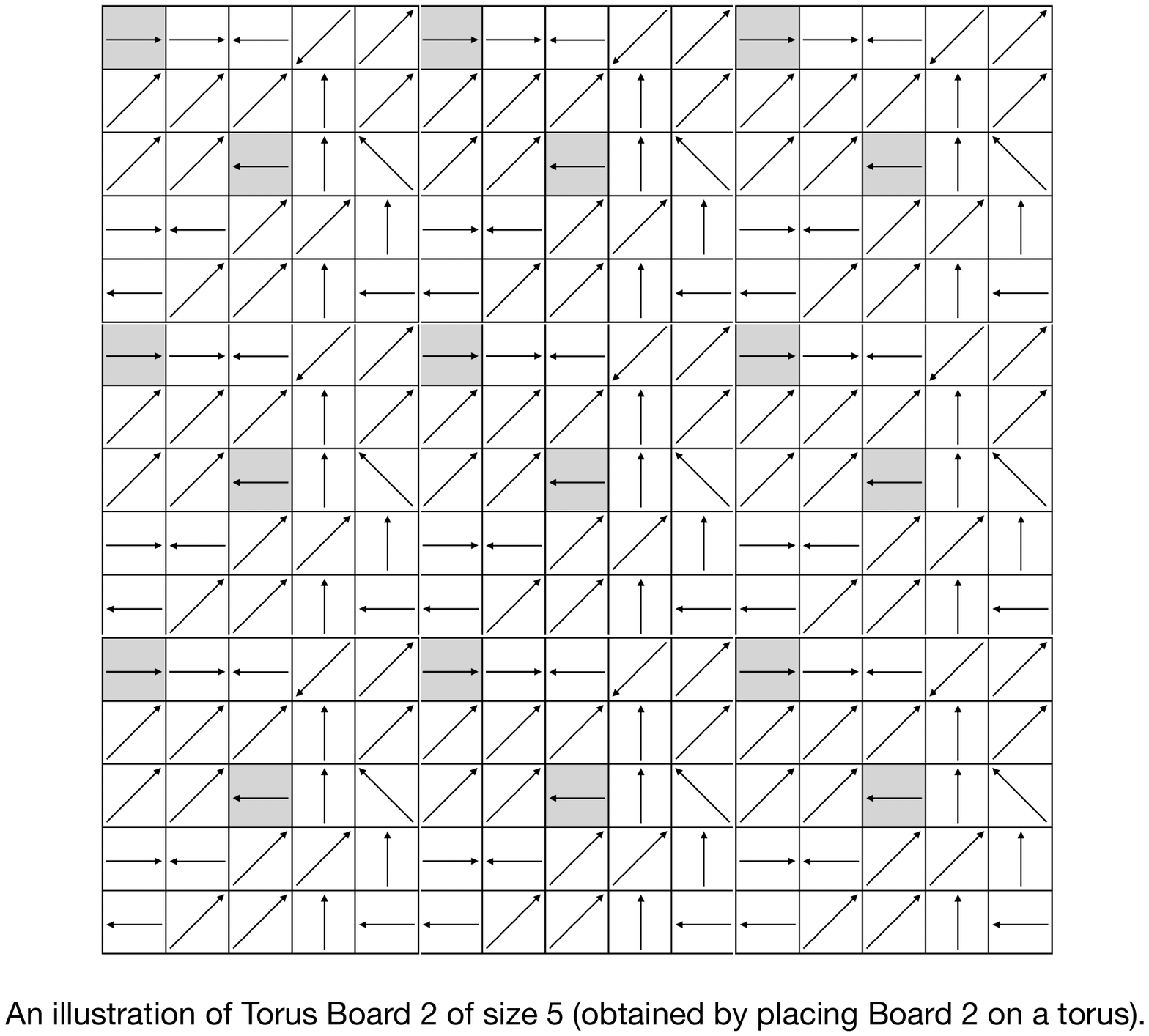}\end{center}

We can now define a torus board of size $n$, in which each arrow is pointing to more entries, e.g., in
Torus Board 2 drawn above the $(1,5)$ position is pointing also to the $(3,3)$ position (we leave it
to the interested reader to write the rigourous definition). We play on this board by again starting from the $(1,1)$ position and our goal is still to arrive at the $(\frac{n+1}{2},\frac{n+1}{2})$ position. Note that though Board 2 is not solvable, Torus Board 2 is solvable. 

We can now ask all the questions we asked regarding usual boards on these torus boards. In particular we may define the length of a solvable torus board, and for torus boards we can prove a strong version of Conjecture \ref{main_conjecture}:

\begin{proposition}\label{proposition_on_ML_on_torus}Fix an odd integer $n\geq3$ and denote by $\overline{ML}(n)$ the maximal length of a solvable torus board of size $n$. Then, $2n-1\leq\overline{ML}(n)\leq4n$ and in particular $\overline{ML}(n)=\Theta(n)$ as $n\to\infty$. \end{proposition}

\begin{proof}First, note that for any $n$ we can generate a torus board using a "spiral pattern" analogues to Board 3. This torus board will have length $2n-1$ and so $2n-1\leq\overline{ML}(n)$.
 
Fix a solvable torus board $A$ of size $n$. There are exactly $4n$ "lines" on $A$, that correspond to its rows, columns, SW-NE diagonals and NW-SE diagonals. 

\begin{center}\includegraphics[width=17cm]{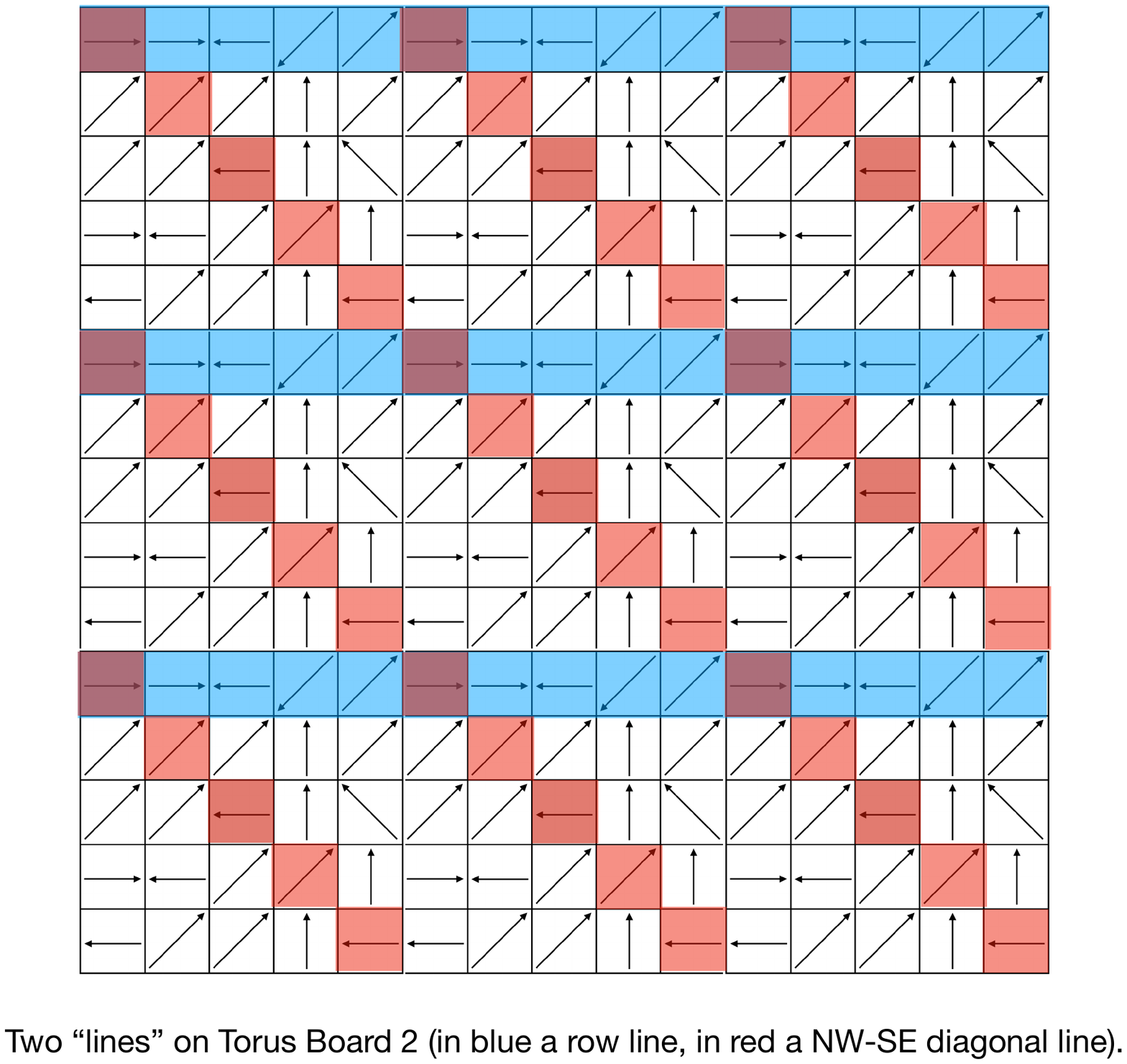}\end{center}

Assume that a specific game $(1,1)=(i_0,j_{0}),(i_1,j_{1}),\dots,(i_k,j_{k})=(\frac{n+1}{2},\frac{n+1}{2})$ is a shortest winning game on $A$. Then, in any transition from $(i_l,j_l)$ to $(i_{l+1},j_{l+1})$ we eliminate one of this "lines", in the sense that we will not visit any position on this "line" in a later time in this game, otherwise we get a contradiction to this game being the shortest possible: indeed, say $(i_{l+l},j_{l+l})$ lies on the same "line" as $(i_l,j_l)$ and $(i_{l+1},j_{l+1})$ for some $l\geq2$. We have that $(i_l,j_l)$ is pointing at $(i_{l+1},j_{l+1})$, but since $(i_{l+l},j_{l+l})$ lies on the same line as $(i_l,j_l)$ and $(i_{l+1},j_{l+1})$ and since this is a torus board, then $(i_l,j_l)$ is also pointing at $(i_{l+l},j_{l+l})$ and we can shorten the game by going directly from $(i_l,j_l)$ to $(i_{l+l},j_{l+l})$.  This implies that $k\leq4n$, i.e., this (arbitrary) board has length at most $4n$, and so we proved that $\overline{ML}(n)\leq4n$. \end{proof}

\subsection{3D boards}One may define a 3-dimensional board in a an analogues way to the way we defined 2-dimensional boards: for an odd $n\geq3$, a 3D board of size $n$ will be a cube in $\mathbb{R}^3$ of side length $n$, while we think about it in the natural way as a collection of $n^3$ unit cubes. Each unit cube contains an arrow, directing to one of the 26 neighbours of this cube (instead of 8 directions of the compass rose here we have 26). The goal is now to arrive from a fixed corner of this 3D board to its center by a path allowed by the arrows. We can now ask all the questions we asked regarding 2D boards on these 3D boards. For instance: what is the probability of a random 3D board to be solvable? What is the expected length of a random 3D solvable board?

The most interesting question is whether the answers to the questions above for 3D boards are \emph{qualitatively} different from the answers we had for 2D boards? One might expect a different behaviour, for instance because of the well known fact (first proven in \cite{polya}) that $\mathbb{Z}^2$ is recurrent whereas $\mathbb{Z}^k$ with $k\geq3$ is transient. For this reason, if indeed qualitatively something different happens in dimension 3, we might expect the same behaviour as in dimension 3 in $k$-dimensional boards when $k\geq4$ (boards in an arbitrary dimension $k$ can easily be defined similarly).

\end{document}